\documentclass[12 pt, a4paper, oneside]{amsart}

\usepackage[utf8]{inputenc}
\usepackage{cmap}
\usepackage[T1]{fontenc}
\usepackage[top=3cm,bottom=2cm,right=2cm,left=3cm]{geometry}
\usepackage{color}
\usepackage{indentfirst}
\usepackage{lmodern} \normalfont
\DeclareFontShape{T1}{lmr}{bx}{sc} { <-> ssub * cmr/bx/sc }{}
\usepackage{amsmath}
\usepackage{amsfonts}
\usepackage{amssymb}
\usepackage{amsrefs}

\usepackage{amsmath}
\usepackage{amsthm,amsfonts,mathrsfs,amsfonts}
\usepackage{slashed}
\usepackage{breqn}
\usepackage{pb-diagram}
\usepackage{bbold}
\usepackage{bbm}
\usepackage[all]{xy}
\usepackage{times}
\usepackage{amsaddr}
\usepackage{microtype}
\usepackage{enumerate}
\usepackage[super]{nth}
\usepackage{multicol}
\usepackage{tikz-cd}

\usepackage{url}

\newtheorem{theorem}{Theorem}[section]
\newtheorem{corollary}[theorem]{Corollary}
\newtheorem{definition}[theorem]{Definition}
\newtheorem{example}[theorem]{Example}
\newtheorem{lemma}[theorem]{Lemma}
\newtheorem{remark}[theorem]{Remark}
\newtheorem*{remark*}{Remark}
\newtheorem{prop}[theorem]{Proposition}

\DeclareMathOperator{\vol}{vol}
\DeclareMathOperator{\Cl}{Cl}
\DeclareMathOperator{\tad}{\widetilde{Ad}}
\DeclareMathOperator{\Gl}{Gl}
\DeclareMathOperator{\SO}{SO}
\DeclareMathOperator{\spin}{Spin}
\DeclareMathOperator{\sptc}{Sp}
\DeclareMathOperator{\im}{Im}
\DeclareMathOperator{\real}{Re}
\DeclareMathOperator{\SU}{SU}
\DeclareMathOperator{\End}{End}
\DeclareMathOperator{\n}{\nabla}
\DeclareMathOperator{\nperp}{\nabla^\perp}
\DeclareMathOperator{\Hom}{Hom}
\DeclareMathOperator{\dg}{deg}

\DeclareMathOperator{\imag}{Im}
\DeclareMathOperator{\ad}{ad}
\DeclareMathOperator{\Span}{Span}
\DeclareMathOperator{\tr}{tr}
\DeclareMathOperator{\Sym}{Sym}
\DeclareMathOperator{\fueter}{\slashed{D}_A}
\DeclareMathOperator{\dirac}{\slashed{D}}
\DeclareMathOperator{\Stab}{Stab}
\DeclareMathOperator{\Aut}{Aut}

\DeclareMathOperator{\Lie}{Lie}
\DeclareMathOperator{\Diff}{Diff}
\DeclareMathOperator{\ident}{Id}

\DeclareMathOperator{\dv}{div}

\newcommand{\R}{\mathbb{R}}
\newcommand{\Langle}{\langle\langle}
\newcommand{\Rangle}{\rangle\rangle}

\newcommand{\gt}{\mathrm{G}_2}

\newcommand{\qforq}{\quad \text{for} \quad}
\newcommand{\qandq}{\quad \text{and} \quad}
\newcommand{\qwithq}{\quad \text{with} \quad}
\newcommand{\qwhereq}{\quad \text{where} \quad}

\makeatletter
\def\@tocline#1#2#3#4#5#6#7{\relax
  \ifnum #1>\c@tocdepth 
  \else
    \par \addpenalty\@secpenalty\addvspace{#2}%
    \begingroup \hyphenpenalty\@M
    \@ifempty{#4}{%
      \@tempdima\csname r@tocindent\number#1\endcsname\relax
    }{%
      \@tempdima#4\relax
    }%
    \parindent\z@ \leftskip#3\relax \advance\leftskip\@tempdima\relax
    \rightskip\@pnumwidth plus4em \parfillskip-\@pnumwidth
    #5\leavevmode\hskip-\@tempdima
      \ifcase #1
       \or\or \hskip 2em \or \hskip 3em \else \hskip 4em \fi%
      #6\nobreak\relax
    \dotfill\hbox to\@pnumwidth{\@tocpagenum{#7}}\par
    \nobreak
    \endgroup
  \fi}
\makeatother

\begin{document}

\title{The Weitzenböck Formula for the Fueter-Dirac Operator}
\author{Andrés J. Moreno \quad \& \quad Henrique N. S\'a Earp}
\address{University of Campinas (Unicamp)}
\date{\today}

\begin{abstract}
We find a Weitzenböck formula for the Fueter-Dirac operator which controls infinitesimal deformations of an associative submanifold in a $7$--manifold
with a $\gt$--structure. We establish a vanishing theorem to conclude rigidity under some
positivity assumptions on curvature, which are particularly mild in the nearly parallel case. As applications, we find a different proof of rigidity for one of Lotay's associatives in the round $7$-sphere from those given by Kawai \cites{kawai2013submanifolds,kawai2014deformations}. We also provide simpler proofs of previous results by Gayet for the Bryant-Salamon metric \cite{gayet2014smooth}. Finally, we obtain an original example  of a rigid associative in a compact manifold with locally conformal calibrated $\gt$-structure obtained by Fernandez-Fino-Raffero \cite{fernandez2016locally}.   
\end{abstract}

\maketitle

%

\tableofcontents

\section*{Introduction}
The theories of Riemannian holonomy and calibrated geometry are related by the fact each Riemannian manifold with  reduced holonomy  is equipped with a calibration.  In particular, a
reduction to the  exceptional holonomy group $\gt$ can only occur in real
dimension $7$, in which case the relevant calibrations are a $3$--form $\varphi$ and its Hodge dual $4$--form $\psi:=\ast\varphi$, and
their calibrated submanifolds  are called \emph{associative} and \emph{coassociative},
respectively (cf. Definition \ref{def: associative submanifold}). In this article, we propose a computational tool to study the deformation theory of associative submanifolds, in  favourable cases of interest.
 
Let $(M^7,\varphi)$ be a smooth manifold with $\gt$--structure. In \cite{mclean1996deformations}, McLean proved that a class in the moduli space of associative deformations corresponds to a harmonic spinor of a twisted Dirac operator, under  the torsion-free hypothesis $\nabla\varphi=0$. Then, Akbulut and Salur \cite{akbulut2008calibrated}, \cite{akbulut2008deformations} generalized McLean's theorem for a general $\gt$-structure, identifying the tangent space at an associative submanifold $Y^3$ in $(M^7,\varphi)$ with the kernel of  
\begin{equation}\label{dirac_operator}
\fueter: \Omega^0(Y,NY)\rightarrow \Omega^0(Y,NY)
\end{equation}
where $A=A_0+a$, for $A_0$  the induced connection on $NY$ and some $a\in \Omega^1(Y,\ad(NY))$.
We obtain a Weitzenböck formula for the operator \eqref{dirac_operator},
that is,  a relation between the second-order elliptic square $\fueter^2$ and the trace Laplacian $\nabla^\ast\nabla$ of the induced Levi-Civita connection
on $NY$. Under suitable positivity assumptions on curvature,  this implies
\emph{rigidity}, i.e., that $Y$\ has ``essentially'' no infinitesimal associative deformations, in the following sense.  Denote by $G:=\Stab(\varphi)\subset\Aut(M)$ the group of global automorphisms preserving $\varphi$. The infinitesimal associative deformations of $Y$ consist of: 
\begin{enumerate}[(i)]
        \item  \emph{trivial} deformations given by the action of $G$ on $Y$ (see  \cite{kawai2014deformations} and \cite{moriyama2013deformations});
        \item  \emph{non-trivial} deformations, which depend intrinsically on the geometry of the associative submanifold.  
\end{enumerate}

For instance, in \cite{kawai2014deformations}, an associative submanifold is considered rigid if all infinitesimal associative deformations are trivial; in the  particular case of the homogeneous space $M=S^7$, the symmetry group of $\varphi$ is $G=\spin(7)$. On the other hand, Gayet \cite{gayet2014smooth} and McLean \cite{mclean1996deformations} consider a generic $\gt$--structure, i.e., without symmetries. So, $G$ is $0$--dimensional and  $Y$ is rigid if the space of nontrivial infinitesimal deformation vanishes. 

The exposition is organised as follows. Section $1$ is a proactive background review. We apply results from $4$-dimensional spin geometry to obtain the explicit identification 
$$
  NY\otimes_{\mathbb{R}}\mathbb{C}\cong S^+\otimes_{\mathbb{C}} S^-,
$$  
between the normal bundle of $Y$ and {a spinor bundle} $S=S^+\oplus S^-\to Y$, in order to describe  the Fueter-Dirac operator in detail. We then  deduce some useful identities in $\rm G_2$--geometry, following Karigiannis  \cite{karigiannis2009flows}.

In Section \ref{sec: General Fueter-Dirac Weitzenbock}, we calculate the general Weitzenböck formula for the operator \eqref{dirac_operator}:
\begin{align}\label{fueterformula}
        \fueter^2(\sigma)
        &=\nabla^\ast\nabla\sigma+\frac{1}{4}k\cdot\sigma+\overline{\rho}(F^-)\sigma
        +P_1(\sigma)+P_2(\sigma)+P_3(\sigma)
\end{align}
where $P_1$, $P_2$ and $P_3$ are first order differential operators on $NY$, involving the torsion of the $\gt$--structure, and $\nabla^\ast\nabla$ is the connection Laplacian 
$$
\nabla^\ast\nabla n=-\sum \n_i^\perp\n_i^\perp n -\n_{\nabla_ie_i}^\perp n
$$
in a global frame  $\{e_i\}$ on the associative submanifold $Y$. The scalar curvature of $Y$ is denoted by $k$, and the bundle map $$\overline{\rho}: \Omega^2(Y, \End(S^-))\rightarrow \End(S^+\otimes S^-)$$is defined by 
\begin{equation}\label{eq: rho barra}
        \overline{\rho}(F^-)
        :=
        \overline{\rho}(\sum (e_i\wedge e_j)\otimes F^-_{ij})=\sum \Gamma_0(e_i)\Gamma_0(e_j)\otimes F^-_{ij},
\end{equation}
where $F^-\in \Omega^2(Y,\End(S^-))$ is the curvature of a connection on $S^-$ and $\Gamma_0: TY\rightarrow \End(S^+)$ is the Spin structure on $Y$.\\
In Section \ref{sec: nearly parallel case}, we specialise to the \emph{nearly parallel} case, in which $d\varphi$ and $\psi$ are collinear and the formula \eqref{fueterformula} simplifies significantly. For a generic nearly parallel $\gt$--structure, we obtain a vanishing theorem (Theorem \ref{rigidity_theorem}) to conclude rigidity under suitable intrinsic geometric conditions on $Y$. As immediate applications,  we propose alternative proofs of rigidity for the known cases of an associative $\SU(2)$-orbit $3$-sphere for Lotay's cocalibrated $\gt$-structure on $S^7$ studied by Kawai \cites{kawai2014deformations,kawai2013submanifolds,lotay2012associative} and 
  the associatives $S^3\times\{0\}$ of the Bryant-Salamon metric studied by Gayet \cite{gayet2014smooth}.

Finally, we obtain a hitherto unstudied rigid associative submanifold (Corollary \ref{cor: rigid assoc Fino-Raffero}) in a compact manifold $S$ with locally conformal calibrated $\gt$-structure obtained from the $3$-dimensional complex Heisenberg group by Fernández-Fino-Raffero  \cite{fernandez2016locally}. In view of the systematic nature of their construction, our method lends itself to the production of many more such examples. 

\bigskip
\noindent\textbf{Acknowledgements.} We are grateful to Selman Akbulut, Jason Lotay and Kotaro Kawai for their constructive input to the preprint version. HSE would like to thank Dietmar Salamon for introducing him to the idea of a Weitzenb\"ock formula in this context, and the Simons Center for Geometry and Physics in Stony Brook, for hosting that conversation in 2014. We thank the anonymous referee for several insightful contributions. 

AM was supported by Coordenação de Aperfeiçoamento de Pessoal de Nível Superior – Brasil (CAPES) – Finance Code 001, grant 1547442, and by Brazilian National Research Council
(CNPq), grant 140689/2017-6. HSE
was supported by S\~ao Paulo Research Foundation (Fapesp), grant 2014/24727-0, and by CNPq Productivity grant 312390/2014-9.

\section{Spin geometry, $\gt$--structures and the Fueter-Dirac operator}\label{G2section}

\subsection{Spin geometry of four dimensional space}

We begin by recalling some background and fixing notation, so the reader
familiar with e.g. \cite[Chapter 2]{salamon2000spin} and  \cite[Chapter 3]{donaldson1990geometry} may just skim through upon a first read. 

On an inner product space  $(V^n,\langle \cdot,\cdot \rangle)$, the Clifford algebra $\Cl(V)$ is a $2^n$-dimensional  associative algebra with unit $1$, generated by the elements of some orthonormal basis $e_1,...,e_n$
of $V$ with relations
$$
e_i^2=-1, \quad e_ie_j=-e_je_i
\qforq
i\neq j.
$$ 
A basis for $\Cl(V)$ is given by $$e_0=1,\quad e_I=e_{i_1}\cdots e_{i_k}$$ where $I=\{i_1,...,i_k\}\subset \{1,...,n\}$ for $i_1<\cdots <i_k$, and
$\Cl(V)$ admits
a natural involution
$$
\alpha: \Cl(V)\rightarrow \Cl(V)
$$
defined by $\alpha(x)=\widetilde{x}:=\sum_I \epsilon_Ix_Ie_I$, where $\epsilon_I:=(-1)^{k(k+1)/2}$ and  $x_I\in \mathbb{R}$ are the components of $x$ in the basis $\{e_I\}$. 
Denote by $\dg(e_I):=|I|$ the degree of an element $e_I\in\Cl(V)$, by $\Cl_k(V)$  the subset of elements of degree $k$, and by $\Cl^0(V)$ and $\Cl^1(V)$ the subspaces of elements of even and odd degree, respectively. 

\begin{example} 
On $V=\mathbb{R}^4$ with the Euclidean inner product, we have $\Cl(V)=M_2(\mathbb{H})$, the $2\times 2$ matrices with entries in the quaternions $\mathbb{H}=\langle i,j,k\rangle$. The elements of $\Cl(V)$ are $1$, $e_i$, $\{e_ie_j\}_{i<j}$, $\{e_ie_je_k\}_{i<j<k}$ and $e_1e_2e_3e_4$, with $i,j,k=1,2,3,4$, with generators

$$e_1=\begin{pmatrix}
0 & 1 \\ -1 & 0
\end{pmatrix}, \, e_2=\begin{pmatrix}
0 & i \\ i & 0
\end{pmatrix}, \, e_3=\begin{pmatrix}
0 & j \\ j & 0
\end{pmatrix} \qandq e_4=\begin{pmatrix}
0 & k \\ k & 0
\end{pmatrix}$$
and the involution $\alpha(A)=A^\ast$ is the transpose conjugation. 
\end{example}
Denote the set of units of $\Cl(V)$ by $\Cl^\times(V)$. Considering the twisted adjoint representation  $\tad:\Cl^\times(V)\rightarrow \Gl(\Cl(V))$ given by 
$$
\tad(x)y=((x)^{0}-(x)^{1})y\widetilde{x},
$$
where $(x)^0\in \Cl^0(V)$ and $(x)^1\in \Cl^1(V)$ are the even and odd parts of $x$, respectively. We define the \emph{Spin group} of $V$:
$$\spin(V):=\{x\in \Cl^{0}(V)| \, \tad(x)V=V, \, x\widetilde{x}=1\}.$$
For $\dim V\geq 3$, $\spin(V)$ is a compact, connected and simply connected Lie group, fitting in a short exact sequence 
\cite[Lemma 4.25]{salamon2000spin} 
$$0\rightarrow \mathbb{Z}_2 \rightarrow \spin(V)\rightarrow \SO(V)\rightarrow 1.$$ 
In particular, the following results hold in  dimensions $3$ and $4$:

\begin{lemma}\cite[Lemma 4.4]{salamon2000spin}
For every $x\in \sptc(1)$, there is a unique orthogonal matrix $\xi_0(x)\in \SO(3)$, such that $\xi_0(x)y=xy\widetilde{x}$, for all $y\in \im(\mathbb{H})\cong \mathbb{R}^3$, and the map $\xi_0: \sptc(1)\rightarrow  \SO(3)$ is a surjective homomorphism with kernel $\{\pm 1\}$, hence 
$$
\SO(3)\cong \sptc(1)/\mathbb{Z}_2 
\qandq 
\spin(3)\cong \sptc(1).
$$  
\end{lemma}

\begin{lemma}\cite[Lemma 4.6]{salamon2000spin}
For every $x,y\in \sptc(1)$, there is a unique orthogonal matrix $\eta_0(x,y)\in \SO(4)$, such that $\eta_0(x,y)z=xz\widetilde{y}$, for all $z\in \mathbb{R}^4\cong \mathbb{H}$, and the map $\eta_0: \sptc(1)\times \sptc(1) \rightarrow  \SO(4)$ is a surjective homomorphism with kernel $\{\pm (1,1)\}$, hence $$\SO(4)\cong \sptc(1)\times \sptc(1)/\mathbb{Z}_2 \ \ \ \text{and} \ \ \ \spin(4)\cong \sptc(1)\times \sptc(1)$$ 
\end{lemma}

The last lemma provides two natural surjective homomorphisms $\rho^\pm: \SO(4)\rightarrow \SO(3)$ and, therefore,  two exact sequences
$$
1\rightarrow \sptc(1) \xrightarrow{\iota^\pm}\SO(4)\xrightarrow{\rho^\pm} \SO(3)\rightarrow 1
$$
where $\iota^+(v)=\eta_0([v,1])$ and $\iota^-(v)=\eta_0([1,v])$, interpreting $\eta_0$ as the induced homomorphism on the quotient $\sptc(1)\times_{\mathbb{Z}_2}\sptc(1)$. Those sequences are related to the $\SO(4)$-action on the spaces of self-dual and anti-self-dual $2$-forms of a $4$-dimensional inner-product space.

An element $q\in \mathbb{H}$ in the canonical basis $q= t+xi+yj+zk=(t+xi)+(y+zi)j$ can be identified with the $2\times 2$ complex matrix 
$$
A=
\begin{pmatrix}
t+xi & -y+zi \\ y+zi & t-xi
\end{pmatrix},
$$
with 
$$
\det A=t^2+x^2+y^2+z^2=|q|^2.
$$
Since $A^\ast A=(\det A)I_2$, every $q\in \sptc(1)\cong S^3$ is identified with a unitary matrix with determinant $1$, that is, $\SU(2)\cong \sptc(1)$.

\begin{definition}
\label{def: spin structure}
Let $V$ be a real inner product space of dimension $2n\equiv 2,4 \mod 8$ or $2n+1\equiv 3 \mod 8$. A \emph{Spin structure} on $V$ is a quadruple $(S,I,J,\Gamma)$, where $S$ is a $2^{n+1}$-dimensional real inner product space, $I$ and $J$ are two anti-commuting orthogonal complex structure 
$$
I^{-1}=I^\ast=-I, 
\quad J^{-1}=J^\ast=-J, 
\quad IJ=-JI,
$$ 
and $\Gamma: V \rightarrow \End(S)$ is a real linear map with the following properties: 
$$
\Gamma(v)^\ast +\Gamma(v)=0, 
\quad \Gamma(v)^\ast\Gamma(v)=|v|^2\mathbb{1}, 
\quad \Gamma(v)I=I\Gamma(v), 
\quad \Gamma(v)J=J\Gamma(v),
\quad\forall v\in V.
$$
\end{definition}

\begin{example}
For a vector space $V$ of real dimension $4$, using the identification $V\cong \mathbb{H}$ and defining $S=\mathbb{H}\oplus \mathbb{H}$, we have the maps $\Gamma: \mathbb{H}\rightarrow \End(\mathbb{H}\oplus \mathbb{H})$, $I,J:\mathbb{H}\oplus \mathbb{H}\rightarrow \mathbb{H}\oplus \mathbb{H}$ defined for  $v,x,y \in \mathbb{H}$ by
$$
\Gamma(v)(x,y)=(vy,-\bar{v}x), 
\quad I(x,y)=(xi,yi), 
\quad J(x,y)=(xj,yj).
$$
It is interesting to note that
$$
 \Gamma(v)=\begin{pmatrix}
            0 & \gamma(v) \\ -\gamma(v)^\ast & 0
           \end{pmatrix},
$$
where $\gamma: \mathbb{H} \rightarrow \End(\mathbb{H})$ also satisfies
$$
\gamma(v)^\ast +\gamma(v)=0, 
\quad \gamma(v)^\ast\gamma(v)=|v|^2\mathbb{1}, 
\quad\forall v\in \mathbb{H}.
$$ 
\end{example}

Given a Spin structure on a $4$-dimensional space $V$, consider $S=S^+\oplus S^-$, where $S^+$ and $S^-$ are copies of $\mathbb{C}^2$ with standard Hermitian metric $\langle \cdot , \cdot \rangle$. The associated symplectic form compatible with the almost complex structure $I: S^\pm \rightarrow S^\pm$ is defined by $\omega(x,y):=\langle x, Iy\rangle.$
Now, consider the  (real) $4$-dimensional space $\Hom_I(S^+,S^-)=\real (\Hom(S^+,S^-))$ of linear maps over the quaternions, where $\Hom(S^+,S^-)$ are complex linear maps. Unitary elements of $\Hom_I(S^+,S^-)$ preserve the Hermitian and symplectic structures,
and  $\gamma: V\rightarrow \Hom_I(S^+,S^-)$ defined above acts on
the standard basis by $$\gamma(e_1)=\begin{pmatrix}
1 & 0 \\ 0 & 1
\end{pmatrix}, \ \ \gamma(e_2)=\begin{pmatrix}
i & 0 \\ 0 & -i 
\end{pmatrix}, \ \ \gamma(e_3)=\begin{pmatrix}
0 & -1 \\ 1 & 0
\end{pmatrix} \ \ \gamma(e_4)=\begin{pmatrix}
0 & i \\ i & 0
\end{pmatrix}.$$
Up to isomorphism, the above generate $\SU(2)\cong \spin(3)$, since the symmetry group  $\SU(2)^+\times\SU(2)^-$
of  $(S^+,S^-)$ is connected. Thus $\gamma$ fixes the orientation of $V$ and, using the sympletic form to identify $S^+$ with its dual, we have 
\begin{equation}\label{iso4}
V\otimes_{\mathbb{R}} \mathbb{C}\cong S^+\otimes_{\mathbb{C}} S^-.
\end{equation}
Moreover, given $v\in V$, consider the Hermitian adjoint $\gamma(v)^\ast : S^-\rightarrow S^+$  of the map $\gamma(v): S^+\rightarrow S^-$. Then, for orthonormal vectors $v,v'\in V$, the map $\gamma(v)^\ast\gamma(v')$ defines an endomorphism of $S^+$ which satisfies
\begin{equation*}
\gamma(v)^\ast \gamma(v)=1
\qandq 
\gamma(v)^\ast\gamma(v')+\gamma^\ast(v')\gamma(v)=0.
\end{equation*}
In particular, we have a natural action $\rho$ of $\Lambda^2(V)$ on $S^+$ defined by 
$$
\rho(v\wedge v')s:=-\gamma(v)^\ast\gamma(v')s 
\qforq 
s\in S^+.
$$ 

Now, with respect to the Euclidean metric, the $2$--forms split as $\Lambda^2(V)=\Lambda^2_+(V)\oplus \Lambda^2_-(V)$, where  $\Lambda^2_+(V)$ and $\Lambda^2_-(V)$ denote the self-dual and anti-self-dual forms, respectively: $$\Lambda^2_\pm(V):=\{\beta \in \Lambda^2(V) \ | \ \ast\beta=\pm\beta\}.$$ We observe that $\Lambda^2_-(V)$ acts trivially on $S^+$, by direct inspection on basis elements: 
$$
\Lambda^2_-(V)= \Span \{e_1\wedge e_2 - e_3\wedge e_4 , e_1\wedge e_4-e_2\wedge e_3, e_1\wedge e_3-e_4\wedge e_2\}
$$ 
\begin{align*}
\rho(e_1\wedge e_2-e_3\wedge e_4) &=-\gamma(e_1)^\ast\gamma(e_2)+\gamma(e_3)^\ast\gamma(e_4)\\
&=\begin{pmatrix}
-1 & 0 \\ 0 & -1 
\end{pmatrix} \begin{pmatrix}
i & 0 \\ 0 & -i
\end{pmatrix}+ \begin{pmatrix}
0 & 1 \\ -1 & 0
\end{pmatrix}
\begin{pmatrix}
0 & i \\ i & 0
\end{pmatrix}= 0,
\end{align*}

\begin{align*}
\rho(e_1\wedge e_4-e_2\wedge e_3) &=-\gamma(e_1)^\ast\gamma(e_4)+\gamma(e_2)^\ast\gamma(e_3)\\
&=\begin{pmatrix}
-1 & 0 \\ 0 & -1 
\end{pmatrix} \begin{pmatrix}
0 & i \\ i & 0
\end{pmatrix}+ \begin{pmatrix}
-i & 0 \\ 0 & i
\end{pmatrix}
\begin{pmatrix}
0 & -1 \\ 1 & 0
\end{pmatrix}= 0,
\end{align*}

\begin{align*}
\rho(e_1\wedge e_3-e_4\wedge e_2) &=-\gamma(e_1)^\ast\gamma(e_3)+\gamma(e_4)^\ast\gamma(e_2)\\
&=\begin{pmatrix}
-1 & 0 \\ 0 & -1 
\end{pmatrix} \begin{pmatrix}
0 & -1 \\ 1 & 0
\end{pmatrix}+ \begin{pmatrix}
0 & -i \\ -i & 0
\end{pmatrix}
\begin{pmatrix}
i & 0 \\ 0 & -i
\end{pmatrix}= 0.
\end{align*}
Thus we get the isomorphisms $\Lambda^2_+(V)\rightarrow \mathfrak{su}(S^+)$ and $\Lambda^2_-(V)\rightarrow \mathfrak{su}(S^-)$.

\subsection{$\gt$--manifolds and associative submanifolds}
We first present some algebraic and geometric proprieties of manifolds with $\gt$--structures which can be found e.g. in \cite{harvey1982calibrated,corti1207g2}. 

The octonions $\mathbb{O}=\mathbb{H}\oplus \mathbb{H}\cong \mathbb{R}^8$ are an $8$-dimensional, non-associative division algebra. On the imaginary part $\imag(\mathbb{O})=\mathbb{R}^7$, the cross product 
$$
\begin{array}{rcccl}
        \times &:& \mathbb{R}^7\times \mathbb{R}^7 &\rightarrow& \mathbb{R}^7\\
        &&(u,v)&\mapsto&\imag(uv)
\end{array}
$$         
corresponds to a $3$-form $\varphi_0\in \Omega^3(\mathbb{R}^7)$, defined by $\varphi_0(u,v,w)=\langle u\times v, w\rangle$ with the Euclidean inner
product. In coordinates $(x_1,...,x_7)\in \mathbb{R}^7$, we fix the convention
\begin{equation}\label{usual_G2-structure}
\varphi_0=e^{123}+e^{145}+e^{167}+e^{246}-e^{257}-e^{347}-e^{356}
\end{equation}
and accordingly its dual $4$--form
\begin{equation*}
\psi_0:=\ast\varphi_0 =e^{4567}+e^{2367}+e^{2345}+e^{1357}-e^{1346}-e^{1256}-e^{1247}.
\end{equation*}
The Lie group $\gt$ can be defined as the stabiliser of  $\varphi_0$ in $\Gl(7,\mathbb{R})$.

\begin{definition}
Let $M$ be a smooth oriented $7$--manifold. A $\gt$--structure is a 3--form $\varphi\in \Omega^3(M)$ such that, around every  $p\in M$, there exists a local section $f$ of the oriented frame bundle P$_{\SO}(M)$ such that $$\varphi_p =(f_p)^\ast\varphi_0.$$
\end{definition}

The $\gt$--structure $\varphi$ determines a Riemannian metric and a volume form by the relation (cf. \cite{karigiannis2009flows})
\begin{equation}\label{G2-metricrelation}
        (u\lrcorner\varphi)\wedge (v\lrcorner\varphi)\wedge\varphi=6g_{\varphi}(u,v)\vol_{\varphi}.
\end{equation}
Consequently, $\varphi$ induces a Hodge star operator $\ast_\varphi$ and the Levi-Civita connection $\nabla^\varphi$, though for simplicity  we omit henceforth the subscripts in $g:=g_\varphi$, $\ast:=\ast_\varphi$ and $\nabla:=\nabla^\varphi$.  Moreover, the model cross-product on $\R^7$ induces the bilinear map on vector fields
\begin{equation} \label{eq: cross-product}
\begin{array}{rcl}
        P:\Omega^0(TM)\times \Omega^0(TM) &\rightarrow& \Omega^0(TM)\\
        (u,v)&\mapsto& P(u,v)=u\times v.
\end{array} 
\end{equation}
The $\gt$--structure $\varphi$ is called \emph{torsion-free} if $\nabla \varphi=0$, in
which case we say that $(M,\varphi)$ is a\emph{ $\gt$--manifold}. This condition is equivalent to $\nabla P=0$.

\begin{remark}\label{sign-convention2}
Regarding orientation conventions, some authors adopt the model $3$-form to be 
$$
  \phi_0=e^{567}+e^{125}+e^{136}+e^{246}+e^{147}-e^{345}-e^{237},
$$
(cf. \cite[Chapters 4 and 5]{mclean1996deformations}), which relates to \eqref{usual_G2-structure} by the orientation-reversing automorphism of $\mathbb{R}^7$
$$\begin{pmatrix}
 & & & & I_3\\
1&0&0&0&    \\
0&1&0&0&    \\
0&0&1&0&    \\
0&0&0&-1&   
\end{pmatrix}.
$$
In this case, relation \eqref{G2-metricrelation} becomes
$$
(u\lrcorner\phi_0)\wedge (v\lrcorner\phi_0)\wedge\phi_0=-6g_0(u,v)\vol_{g_0}.
$$
Unless otherwise stated, we adopt throughout the convention \eqref{usual_G2-structure}.
\end{remark}

\begin{definition}
\label{def: associative submanifold}
Let $(M,\varphi)$ be a $7$--manifold with  $\gt$--structure. A $3$--dimensional submanifold $Y\subset M$ is called \emph{associative} if $\varphi|_Y\equiv \vol(Y)$.
\end{definition}

The etymology of Definition \ref{def: associative submanifold} stems from
the  \emph{associator} $\chi\in \Omega^3(M,TM)$, defined by 
\begin{equation}\label{psi_associator}
\psi(u,v,w,z)=\ast\varphi(u,v,w,z)=\langle \chi(u,v,w),z\rangle.
\end{equation}
In a local orthonormal frame   $\{e_i\}_{i=1\dots 7}$ of $TM$, one has $\chi=-\sum_{i=1}^7(e_i\lrcorner \psi)\otimes e_i$. Expressing $\chi$ in terms of the cross product (c.f. \cite{harvey1982calibrated}),
\begin{equation}\label{def_associator}
\chi(u,v,w)=-u\times (v\times w)-\langle u,v\rangle w+\langle u,w\rangle v,
\end{equation}
and using the relation (c.f.\cite{harvey1982calibrated})
\begin{equation*}
\varphi(u,v,w)+\frac{1}{4}|\chi(u,v,w)|^2=|u\wedge v\wedge w|^2,
\end{equation*}we see that the associative condition is equivalent to $\chi|_Y \equiv 0$.

\begin{remark}
In  the sign convention of Remark \ref{sign-convention2}, the associator is written as
\begin{equation*}
\chi(u,v,w)=u\times (v\times w)+\langle u,v\rangle w-\langle u,w\rangle v
\end{equation*}
\end{remark}

\begin{lemma}   \label{lemma: TY=Lambda+}
If $Y$ is an associative submanifold, then there is a natural identification $TY\cong \Lambda^2_+(NY)$.
\end{lemma}

\begin{proof}
Fix local orthonormal frames $e_1,e_2,e_3$ and $\eta_4,\eta_5,\eta_6,\eta_7$ of  $TY$ and $NY$, respectively, about a point $p\in Y$: 
\begin{equation}\label{varphi_2}
\varphi_p=e^{123}+e^1(\eta^{45}+\eta^{67})+e^2(\eta^{46}+\eta^{75})-e^3(\eta^{47}+\eta^{56})
\end{equation}
and 
\begin{align*}
e_1\lrcorner \varphi &= e^{23}+\eta^{45}+\eta^{67},\\
e_2\lrcorner \varphi &= e^{31}+\eta^{46}+\eta^{75},\\
e_3\lrcorner \varphi &= e^{12}-\eta^{47}-\eta^{56}.
\end{align*}
Denote $\omega_1=(e_1\lrcorner \varphi)|_{N_pY}$, $\omega_2=(e_2\lrcorner \varphi)|_{N_pY}$, $\omega_3=-(e_3\lrcorner \varphi)|_{N_pY}$ and define on each fibre the isomorphism  $e_j\in T_pY\mapsto \omega_j\in \Lambda_+^2(N_pY)$,
which obviously varies smoothly with $p$. 
\end{proof}

\subsection{The twisted Dirac operator}
The oriented orthonormal frame of $TY$ has the form $\{e_1,e_2,e_3=e_1\times e_2\}$. So, with respect to the splitting $TM|_Y=TY\oplus NY$, the cross product induces maps
$$
\Omega^0(TY)\times \Omega^0(TY)\rightarrow \Omega^0(TY), \quad \Omega^0(TY)\times \Omega^0(NY)\rightarrow \Omega^0(NY)$$
$$\Omega^0(NY)\times \Omega^0(NY)\rightarrow \Omega^0(TY).$$
In particular, the map $\gamma:  \Omega^0(TY)\times \Omega^0(NY)\rightarrow \Omega^0(NY)$ endows $NY$ with a Clifford bundle structure.\\
Since the Levi-Civita connection of $(M,\varphi)$ induces metric connections on the bundles $TY$ and $NY$,  the composition 
\begin{equation}        \label{diraccomposition}
        \Omega^0(NY)\xrightarrow{\nabla_{A_0}} \Omega^0(TY)\otimes\Omega^0(NY)\xrightarrow{\gamma}\Omega^0(NY)
\end{equation}
defines a natural \emph{Fueter-Dirac operator} $\dirac_{A_0}(\sigma):=\gamma(\nabla_{A_0}(\sigma))$, where $A_0\in\Omega^1(Y,\mathfrak{so}(4))$ denotes the connection induced on $NY$ by the Levi-Civita connection $\nabla^\varphi$ of the $\gt$-metric of $(M,\varphi)$. To simplify the notation, the twisted Dirac operator induced by the normal connection $A_0$ will be denoted just by $\dirac$.\\ 
        
The normal bundle $NY$ of an associative submanifold is trivial  \cite[Lemma
5.1, arXiv version: 1207.4470v3]{corti1207g2}. In particular, the second Stiefel-Whitney class $w_2(NY)$ vanishes, so there exists a spin structure on $NY$ \cite[Theorem 1.7]{lawson1989spin}. This is equivalent to the existence of a map $\Gamma: NY\rightarrow \End(S)$ such that 
$$\Gamma(\sigma)+ \Gamma(\sigma)^\ast=0 \ \ \ \Gamma(\sigma)^\ast\Gamma(\sigma)=\langle\sigma,\sigma\rangle\mathbb{1} \ \ \ \sigma\in  \Omega^0(Y,NY),$$
where $S$ is a vector bundle of (real) rank $8$ and it splits into $\Gamma$--eigenbundles $S^+$ and $S^-$ of rank $4$. We saw in the last Section that the Spin structure induces an isomorphism $$\rho_\pm: \Lambda^2_\pm (NY)\rightarrow \mathfrak{su}(S^\pm),$$so, by Lemma \ref{lemma: TY=Lambda+}, the Spin structure $\Gamma_0: TY\rightarrow \End(S^+)$ on $TY$ coincides with the Spin structure on $NY$ via the projection $\spin(4)=\spin(3)\times\spin(3)$. Defining the Clifford multiplication 
$$
\tau :=\Gamma_0\otimes \mathbb{1}_{S^-} 
        \colon 
        TY  \rightarrow  \End(S^+\otimes S^-)
$$
and using the Spin connection $\nabla$ on $S^+\otimes S^-$, 
$$
\nabla(\sigma\otimes \varepsilon)=\nabla^+\sigma \otimes \varepsilon +\sigma\otimes \nabla^-\varepsilon,
$$
we form the Dirac operator $D: \Omega^0(Y,S^+ \otimes S^-)\rightarrow \Omega^0(Y,S^+\otimes S^-)$ by
$$
D(\sigma\otimes\varepsilon):=\sum_{i=1}^3 \tau(e_i)\nabla_i(\sigma\otimes \varepsilon).
$$

\begin{prop} \label{prop: NY=S+ x S-}
Under the isomorphism \eqref{iso4}, we have $NY\otimes_{\mathbb{R}} \mathbb{C}\cong S^+\otimes_{\mathbb{C}} S^-$, the Spin connection $\nabla$ and the Clifford
multiplication $\tau$ agree with the induced connection $\nperp$  on  $NY$ and  $\gamma$, respectively. 
\begin{proof}
In fact, each section $\sigma\otimes \varepsilon$ of $S^+\otimes_{\mathbb{C}}S^-$ induces a section $\nu=\sigma^\ast\otimes \varepsilon$ on $\Hom(S^+,S^-)\cong (S^+)^\ast \otimes S^-$ such that $\nu(\sigma)=\sigma^\ast(\sigma)\otimes\varepsilon=\varepsilon$, then
\begin{align*}
        \nabla \nu 
        &= \nabla(\sigma^\ast \otimes \varepsilon)\\
        &= (\nabla^+)^\ast \sigma^\ast \otimes \varepsilon+\sigma^\ast \otimes \nabla^-  \varepsilon,
\end{align*}
where $\nabla \nu$ is a section on $T^\ast Y\otimes \Hom(S^+,S^-)$, so, for each $\sigma$ section on $S^+$
\begin{align*}
        (\nabla \nu)(\sigma) 
        &=(\nabla^+)^\ast \sigma^\ast (\sigma) \otimes \varepsilon+\sigma^\ast (\sigma) \otimes \nabla^-  \varepsilon\\
        &=[d\sigma^\ast (\sigma)-\sigma^\ast(\nabla^+ \sigma)]\otimes \varepsilon+\sigma^\ast(\sigma)\otimes \nabla^- \varepsilon\\
        &= -\nu(\nabla^+\sigma)+\nabla^-(\nu(\sigma)).
\end{align*}
On the other hand, the Spin connection $\nabla$ is compatible with the induced connection $\nperp$, that is,
$$
\nabla^-(\Gamma(n)\sigma)=\Gamma(\nperp n)\sigma+\Gamma(n)\nabla^+\sigma,
$$
where $\Gamma: NY\rightarrow \Hom_J(S^+,S^-)$ is the isomorphism induced by \eqref{iso4}, then for each section $n$ of $NY$ and $\sigma$ of $S^+$,
$$
\Gamma(\nperp n)=-\Gamma(n)\nabla^+ \sigma+\nabla^-(\Gamma(n)\sigma).
$$
Therefore, $\nperp$ agrees with the Spin connection $\nabla$ via the isomorphism $\Gamma$.
Finally, with respect to the Clifford multiplications we have 

\begin{center}
\begin{tikzcd}
TY \arrow{r}{\Gamma_0} \arrow{d}{\gamma} & \End(S^+)\arrow[hook]{rr}{\otimes \mathbb{1}_{\End(S^-)}} & &\End(S^+\otimes_{\mathbb{C}} S^-) \\
\End(NY\otimes_{\mathbb{R}}\mathbb{C})\arrow[swap]{urrr}{\cong}&&&
\end{tikzcd}
\end{center}
and by Schur's lemma $\gamma$ and $\tau$ are the same.
\end{proof}

\end{prop}

 In conclusion,  \eqref{diraccomposition} defines a twisted Dirac operator.

\subsection{Torsion tensor and local description of $\varphi$}
\label{sec: torsion tensor}

We will briefly review the intrinsic torsion forms of a $\gt$--structure and define the full torsion tensor $T_{ij}$, using local coordinates, following  \cite{karigiannis2009flows}.  Our goal is to derive Lemma \ref{Leibniz_rule}, a set of `Leibniz rules' for the covariant derivative and curvature operators with respect to the vector cross-product, which will be instrumental in Section \ref{sec: General Fueter-Dirac Weitzenbock}.

As before, let $(M,\varphi)$ be a smooth $7$--manifold with $\gt$--structure. In a local coordinate system $(x_1,...,x_7)$, a differential $k$--form $\alpha$ on $M$ will be written as 
$$
  \alpha=\frac{1}{k!}\alpha_{i_1\cdots i_k}dx^{i_1}\wedge\cdots\wedge dx^{i_k}
$$ 
where the sum is taken over all ordered subsets $\{i_1\cdots i_k\}\subset \{1,...,7\}$ and $\alpha_{i_1\cdots i_k}$ is skew-symmetric in all indices, i.e. $\alpha_{i_1\cdots i_k}=\alpha(e_{i_1},...,e_{i_k})$. A Riemannian metric $g$ on $M$ induces on $\Omega^k:=\Omega^k(M)$ the metric  $g(dx^i,dx^j):=g^{ij}$, where $(g^{ij})$ denotes the inverse of the matrix $(g_{ij})$.

A $\gt$--structure $\varphi$ splits $\Omega^\bullet$ into orthogonal irreducible $\gt$ representations, with respect to its $\gt$--metric $g$. In particular, 
$$
\Omega^2= \Omega^2_7\oplus \Omega^2_{14}
\qandq
\Omega^3 =\Omega^3_1\oplus\Omega^3_7\oplus\Omega^3_{27},
$$  
where $\Omega_l^k \subset\Omega^k$ denotes (fibrewise) an irreducible $\gt$--submodule of dimension $l$, with an explicit  description:
\begin{align}\label{decomposition_spaces}
\begin{split}
\Omega^2_7 &=\{X\lrcorner \varphi; X\in \Omega^0(TM)\}\\
\Omega^2_{14} &=\{\beta\in \Omega^2; \beta\wedge\psi=0\}\\
\Omega^3_1 &=\{f\varphi; f\in C^\infty(M)\}\\
\Omega^3_7 &=\{X\lrcorner \psi; X\in\Omega^0(TM)\}\\
\Omega^3_{27} &=\{h_{ij}g^{jl}dx^i\wedge\biggl(\frac{\partial}{\partial x_l}\biggl)\lrcorner \varphi; h_{ij}=h_{ji}, \tr_g(h_{ij})=g^{ij}h_{ij}=0\}
\end{split}
\end{align}

The analogous decompositions of $\Omega^4$ and $\Omega^5$ are obtained from
the above by the Hodge isomorphism $\ast_\varphi:\Omega^k\to\Omega^{7-k}$.
Decomposing $d\varphi\in \Omega^4$ and $d\psi\in \Omega^5$, we introduce
the four \emph{torsion forms}  (cf. \cite{bryant2003some})
\begin{equation*}
        \tau_0\in\Omega^0_1,
        \quad \tau_1\in \Omega^1_7,
        \quad \tau_2\in \Omega^2_{14},
        \quad\tau_3\in \Omega^3_{27},
\end{equation*} 
defined by
\begin{equation*}
        d\varphi=\tau_0\psi+3\tau_1\wedge\varphi+\ast\tau_3
        \qandq d\psi=4\tau_1\wedge\psi+\ast\tau_2.
\end{equation*}
Naturally, these forms arise from the theorem of Fernandez and Gray \cite{fernandez1982riemannian}, asserting that a $\gt$--structure is torsion-free ($\nabla\varphi=0$) if, and only if, $\varphi$ is closed  and co-closed. So, if either condition fails,  the torsions $\nabla\varphi\in \Omega^1\otimes\Omega^3$ and $\nabla\psi\in \Omega^1\otimes\Omega^4$ can
be expressed in terms of the four torsion forms.

\begin{lemma}[\cite{karigiannis2009flows}, Lemma 2.24]\label{Omega3}
For any vector field, the $3$-form $\nabla_X\varphi$ lies in the subspace $\Omega^3_7$ of $\Omega^3$.
\end{lemma}
\begin{proof}
It suffices to consider a coordinate vector $X=e_l$ and check that $g(\nabla_l\varphi,\eta)=0$ for an arbitrary $\eta\in \Omega^3_1\oplus\Omega^3_{27}$.
\end{proof}

In a local frame  $\{e_1,\dots,e_7\}$, denoting  $\nabla_l\varphi:=\frac{1}{6} \nabla_l\varphi_{abc}dx^a\wedge dx^b\wedge dx^c$ and identifying $\Omega^3_7\cong \Omega^1$, we see from \eqref{decomposition_spaces} that $\nabla_l\varphi$ is spanned by interior products $e_n\lrcorner\psi$, which defines a $2$--tensor $T_{lm}$ by 
$$
\nabla_l\varphi_{abc}=:T_{lm}g^{mn}\psi_{nabc}
$$ 
called the \emph{full torsion tensor}.

\begin{prop}[\cite{karigiannis2009flows}, Theorem 2.27]
\label{prop: full torsion tensor}
The full torsion tensor $T_{lm}$ is 
\begin{equation}
T_{lm}=\frac{\tau_0}{4}g_{lm}-(\tau_3)_{lm}+(\tau_1)_{lm}-\frac{1}{2}(\tau_2)_{lm},
\end{equation}
where $\tau_0$ is a function, $g_{lm}=g(e_l,e_m)$, $\tau_1=(\tau_1)_ldx^l$ is $\Omega^1_7$-form which can be written as a $\Omega^2_7$-form $\tau_1=\frac{1}{2}(\tau_1)_{ab}dx^a\wedge dx^b$ with $(\tau_1)_{ab}=(\tau_1)_lg^{lk}\varphi_{kab}$, $\tau_2=\frac{1}{2}(\tau_2)_{ab}dx^a\wedge dx^b$ and $\tau_3=\frac{1}{2}(\tau_3)_{im}g^{ml}\varphi_{ljk}dx^i\wedge dx^j\wedge dx^k$ is a $\Omega^3_{27}$-form. 
\end{prop} 

In \cite[Lemma A.14]{karigiannis2009flows}, Karigiannis   compiles several useful identities among the tensors $g$, $\varphi$ and $\psi$:
\begin{align}
        \psi_{rstu}\psi_{abcd}g^{ra}g^{sb}g^{tc}g^{ud}=&168
        \label{contraction_psi}\\
        \psi_{rstu}\psi_{abcd}g^{sb}g^{tc}g^{ud}=&24g_{ra}
        \label{contraction_psi2}
\end{align} 
Differentiating \eqref{contraction_psi} and \eqref{contraction_psi2}, one obtains 
\begin{equation}\label{first_covariant_psi}
        \nabla_l\psi_{rstu}\psi_{abcd}g^{ra}g^{sb}g^{tc}g^{ud}=0, 
\end{equation}
\begin{equation}\label{second_covariant_psi}
        \nabla_l\psi_{rstu}\psi_{abcd}g^{sb}g^{tc}g^{ud}=-\psi_{rstu}\nabla_l\psi_{abcd}g^{sb}g^{tc}g^{ud}.
\end{equation} 

\begin{lemma}\label{lemma: covariant_psi2}
For any vector field $X$, the $4$-form $\nabla_X\psi$ lies in the subspace $\Omega^4_7$ of $\Omega^4$.
\end{lemma}

\begin{proof}
The proof is, basically, the same as in Lemma \ref{Omega3}. Considering $X=e_l$
and applying \eqref{first_covariant_psi}, we have
$$g(\nabla_l\psi,\psi)=\frac{1}{24}\nabla_l\psi_{rstu}\psi_{abcd}g^{ra}g^{sb}g^{tc}g^{ud}=0,$$
so $\nabla_l\psi\perp\Omega^4_1$. To see that   $\nabla_l\psi\perp \Omega^4_{27}$,
 consider some $\eta\in \Omega^4_{27}$  in local form, 
$$
\eta=\frac{1}{3!}h_{ij}g^{jl}\psi_{labc}dx^i\wedge dx^a\wedge dx^b\wedge dx^c,
$$
and take the inner product with $\nabla_l\psi$:
$$
g(\nabla_l\psi,\eta)=\frac{1}{3!}\nabla_l\psi_{rstu}h_i^l\psi_{labc}g^{ri}g^{sa}g^{tb}g^{uc}
=\frac{1}{3!}h^{lr}\nabla_l\psi_{rstu}\psi_{labc}g^{sa}g^{tb}g^{uc}
=0,
$$
using that, $h^{lr}=g^{ri}h_{ij}g^{jl}$ is a symmetric $(0,2)$-tensor (since $h_{ij}$ is a symmetric $(2,0)$-tensor), while    $\nabla_l\psi_{rstu}\psi_{labc}g^{sa}g^{tb}g^{uc}$ is skew-symmetric in $r$ and $l$, by \eqref{second_covariant_psi}. 
\end{proof}
Using Lemma \ref{lemma: covariant_psi2} above and the identity $\ast(X\lrcorner \psi)=\varphi\wedge X^\flat$ ( $X\in \Omega^0(M)$ ), where $X^\flat$ is the $1$--form defined by $X^\flat(Y)=g(X,Y)$, one has: 

\begin{corollary}\label{covariant_psi}
With the above notation,
        $$\nabla_l\psi_{rstu}=-T_{lr}\varphi_{stu}+T_{ls}\varphi_{rtu}-T_{lt}\varphi_{rsu}+T_{lu}\varphi_{rst}.$$

\end{corollary}
For a torsion-free $\gt$--structure, the cross-product (\ref{eq: cross-product}) is parallel, so it satisfies the Leibniz rule    
$$
\nabla (u\times v)=\nabla u\times v+u\times \nabla v, 
\quad\forall u,v\in \Omega^0(TM).
$$ 
In general, the action of $\nabla$ on the cross product can be expressed in terms of the total torsion tensor:

\begin{lemma}\label{Leibniz_rule}
For the vector fields $u,v,w,z \in \Omega^0(TM)$, we have
\begin{enumerate}
        \item[\rm{(i)}]
        $\displaystyle\nabla_z(u\times v)=\nabla_z u\times v+u\times\nabla_z v+\sum_{m=1}^7T(z,e_m)\chi(e_m,u,v)$.

        \item[\rm{(ii)}]
        $R(w,z)(u\times v)=R(w,z)u\times v+u\times R(w,z)v+\mathcal{T}(w,z,u,v)$, where
\begin{align}\label{torsion_four_tensor}
\begin{split}
            \mathcal{T}(w,z,u,v)
            &:=\sum_{m=1}^7T(z,e_m)(\nabla_w\psi)(e_m,u,v,\cdot)^\sharp-T(w,e_m)             (\nabla_z\psi)(e_m,u,v,\cdot)^\sharp\\
            &\qquad+\Bigl((\nabla_wT)(z,e_m)-(\nabla_zT)(w,e_m)\Bigr)\chi(e_m,u,v)
\end{split}            
\end{align}
in an orthonormal local frame $\{e_1,...,e_7\}$  of $TM$.

\end{enumerate}

\begin{proof}
\begin{enumerate}[(i)]
        \item[(i)] 
        The proof goes along the lines of \cite[Lemma A.1]{gayet2014smooth}, using the fact that the torsion $\nabla\varphi$ takes values in $\Omega^3_7$ (c.f. Lemma \ref{Omega3}). Consider normal coordinates $x_1,...,x_7$  about a given $p\in M$ and an orthonormal frame $e_1,...,e_7$. At the point $p$, we have:
\begin{align*}
     \nabla_z(u\times v)&=\sum_{i=1}^7\nabla_z(\langle u\times v,e_i\rangle e_i)=\sum_{i=1}^7\nabla_z(\varphi(u,v,e_i)e_i)\\
                        &=\sum_{i=1}^7 z(\varphi(u,v,e_i))e_i +\varphi(u,v,e_i)\nabla_ze_i\\ &=\sum_{i=1}^7 \Big(\varphi(\nabla_z u,v,e_i)+\varphi(u,\nabla_z v,e_i)+\varphi(u,v,\nabla_ze_i)+(\nabla_z\varphi)(u,v,e_i)\Big)e_i\\
                        &=\sum_{i=1}^7 \biggl(\varphi(\nabla_z u,v,e_i)+\varphi(u,\nabla_z v,e_i)+
     \sum_{m=1}^7T(z,e_m)\psi(e_m,u,v,e_i)\biggr)e_i\\
                        &=\nabla_zu\times v+u\times\nabla_zv+\sum_{m=1}^7T(z,e_m)\chi(e_m,u,v).
\end{align*}
Notice that we used $(\nabla_j e_i)_p=0$ in the third and fourth equalities.
        \item[(ii)]
        Using the first part, we have
\begin{eqnarray*}
     \nabla_w\nabla_z(u\times v)
     &=&\nabla_w\nabla_zu\times v+\nabla_zu\times\nabla_wv+\nabla_wu\times
                                 \nabla_zv+u\times\nabla_w\nabla_zv\\
     &&+\sum_{i,m=1}^7\biggl(T(w,e_m)\Bigl(\psi(e_m,\nabla_zu,v,e_i) 
                                 +\psi(e_m,u,\nabla_zv,e_i)\Bigr)\\
     &&\qquad+\Bigl((\nabla_wT)(z,e_m)+T(\nabla_wz,e_m)\Bigr)\psi(e_m,u,v,e_i)\\
     &&\qquad+T(z,e_m)\Bigl(\psi(e_m,\nabla_wu,v,e_i)+\psi(e_m,u,\nabla_wv,e_i)\\
     &&\qquad+(\nabla_w\psi)(e_m,u,v,e_i)\Bigr)\biggr)e_i.
\end{eqnarray*}
Using the symmetries of the curvature tensor 
$R(w,z)=\nabla_w\nabla_z-\nabla_z\nabla_w-\nabla_{[w,z]}$ and the fact that $\nabla$ is torsion-free, one has $[w,z]=\nabla_wz-\nabla_zw$, and we compute
\begin{align*}
     R(w,z)(u\times v)
     \quad=\quad & R(w,z)u\times v+u\times R(w,z)v\\
     &   +\sum_{i,m=1}^7\biggl(T(z,e_m)(\nabla_w\psi)
                       (e_m,u,v,e_i)\\
     &\qquad+\Bigl((\nabla_wT)(z,e_m)-(\nabla_zT)(w,e_m)\Bigr)\psi(e_m,u,v,e_i)\\
     &\qquad-T(w,e_m)(\nabla_z\psi)(e_m,u,v,e_i)\biggr)e_i \qedhere
\end{align*}
\end{enumerate}
\end{proof}
\end{lemma}

\section{The Fueter-Dirac Weitzenb\"ock formula}
\label{sec: General Fueter-Dirac Weitzenbock}

We now address the general framework proposed by Akbulut and Salur \cite{akbulut2008calibrated,akbulut2008deformations},
in which the role of torsion in the associative deformation theory is captured by a \emph{twisted} Fueter-Dirac operator. 
Given an associative submanifold  $Y^3$ in $ (M,\varphi)$, the $\gt$--structure induces connections on the bundles $NY$ and $TY$. Moreover, Proposition \ref{prop: NY=S+ x S-} gives an identification $NY\cong \real(S^+\otimes_{\mathbb{C}}S^-)$, with the respective reductions $\Lambda^2_\pm(NY)\cong \mathfrak{su}(S^\pm)=\ad(S^\pm)$. We will refer to elements in the kernel  $\ker \dirac$  of the Dirac operator (\ref{diraccomposition}) as harmonic spinors twisted by $S^-$, or simply, \emph{twisted harmonic spinors}.

Denote by $\mathcal{A}(S^\pm)$ the space of connections on each spinor bundle $S^\pm$, and let $A_0\in \Omega^1(Y,\mathfrak{so}(4))$ be the induced connection on  $NY$, so that
the  isomorphism $\mathfrak{so}(4)\cong \mathfrak{so}(3)\oplus \mathfrak{so}(3)$ gives a decomposition $A_0=A_0^+\oplus A_0^-$, with $A_0^\pm \in \mathcal{A}(S^\pm)$.
 Fixing these reference connections,  each $\mathcal{A}(S^\pm)$ is an affine space modelled on $\Omega^1(Y,\ad(S^\pm))$, so a connection $A^\pm \in \mathcal{A}(S^\pm)$ is of the  form 
$$
A^\pm=A_0^\pm+a^\pm 
\qforq
a^\pm \in \Omega^1(Y,\ad(S^\pm)).
$$ 
Thus a connection on $NY$ has the form
$$A=A_0+a=(A_0^++a^+)\oplus (A_0^-+a^-) \qforq a\in \Omega^1(Y,\ad(NY)).
$$ 
Now, using the Clifford multiplication (indeed the cross-product), we define the \emph{twisted Dirac operator}
\begin{eqnarray*}
        \fueter:=\sum_{j=1}^3 e_i\times \nabla_{e_i}
        &\colon&
         \Omega^0(NY)\rightarrow \Omega^0(NY)  
\end{eqnarray*}
where $\nabla:=\nabla_A$ is given by a connection on $NY$ and the normal
sections in $\ker (\fueter)$ are called \emph{harmonic spinors twisted by $(S^-,A)$}. The following Definition is adopted from \cite{akbulut2008calibrated}:

\begin{definition}
Let $Y$ be an associative submanifold of $(M,\varphi)$. The \emph{Fueter-Dirac operator} associated with $Y$ is
\begin{equation}\label{fueter_dirac_operator}
\fueter\sigma:= \sum_{i=1}^3 e_i\times \n_{e_i}^\perp \sigma -e_i\times a(e_i)(\sigma),
\end{equation}
where $a\in \Omega^1(Y,\ad(NY))$ defined by $a(e_i)(\sigma)=(\nabla_\sigma (e_i))^\perp$ is the normal component of $\nabla_\sigma(e_i)$, and $\nabla$ is the Levi-Civita connection on $M$.
\end{definition}   

We know from \cite[Theorem 6]{akbulut2008calibrated} that the linearisation of the deformation problem for an associative submanifold $Y$ of $(M,\varphi)$ at  $Y$ is identified with $\ker \fueter$, so this space is called the \emph{infinitesimal deformation space} of $Y$.
Our motivation is precisely the expectation that a Weitzenböck formula for \eqref{fueter_dirac_operator}, in favourable cases at least, can give  information about the deformation space $\ker \fueter$.

\subsection{The general squared Fueter-Dirac operator}

\begin{lemma}
Let $\{e_1,e_2,e_3\}$ and $\{\eta_4,...,\eta_7\}$ be orthonormal frames of the vector bundles $TY$ and $NY$, respectively. Then
\begin{equation}\label{D_with_torsion}
        \fueter\sigma
        =
        \sum_{i=1}^3 e_i\times \n_{e_i}^\perp \sigma 
        - 
        \sum_{k=4}^7(\nabla_{\sigma} \psi)(\eta_k,e_1,e_2,e_3)\eta_k
\end{equation}
\end{lemma}

\begin{proof}
Since $A_0$ is the connection induced on $NY$ by the Levi-Civita connection on $M$ given by the $\gt$ metric $g_{\varphi}$, we have $\nabla_{A_0}=\nperp$. Now, for each $\sigma\in \Omega^0(NY)$, 
\begin{align*}
        \sum_{i=1}^3e_i\times a(e_i)(\sigma)
        &= e_1\times (\nabla_{\sigma}e_1)^\perp +e_2\times (\nabla_{\sigma}e_2)^\perp +e_3\times (\nabla_{\sigma}e_3)^\perp\\
        &=(e_2\times e_3)\times (\nabla_{\sigma}e_1)^\perp +(e_3\times e_1)\times (\nabla_{\sigma}e_2)^\perp +(e_1\times e_2)\times (\nabla_{\sigma}e_3)^\perp\\
        &=\chi((\nabla_{\sigma}e_1)^\perp,e_2,e_3)+\chi((\nabla_{\sigma}e_2)^\perp,e_3,e_1)+\chi((\nabla_{\sigma}e_3)^\perp,e_1,e_2)\\
        &= (\diamondsuit).
\end{align*}
Since $Y$ is associative exactly when $\chi|_{TY}=0$, this implies 
$$
\chi((\nabla_{\sigma}e_i)^\perp,e_j,e_k)=\chi(\nabla_{\sigma}e_i,e_j,e_k).
$$
Furthermore, the section $\chi(\nabla_{\sigma}(e_i),e_j,e_k)$ lies on the normal component, so
\begin{align*}
(\diamondsuit) 
        &=\sum_{k=4}^7 (\langle \chi(\nabla_{\sigma}(e_1),e_2,e_3),\eta_k\rangle+\langle \chi(e_1,\nabla_{\sigma}(e_2),e_3),\eta_k\rangle+\langle\chi(e_1,e_2,\nabla_{\sigma}(e_3)),\eta_k\rangle )\eta_k \\
        &=\sum_{k=4}^7(-(\nabla_{\sigma}\psi)(e_1,e_2,e_3,\eta_k)+\sigma(\psi(e_1,e_2,e_3,\eta_k))-\psi(e_1,e_2,e_3,\nabla_{\sigma}(\eta_k)))\eta_k\\
        &=\sum_{k=4}^7((\nabla_{\sigma}\psi)(\eta_k,e_1,e_2,e_3)\eta_k. 
\end{align*}
To obtain the second equality we used the covariant derivative of $\psi$:
$$
  (\nabla_{\sigma}\psi)(e_1,e_2,e_3,\eta_k)=\sigma(\psi(e_1,e_2,e_3,\eta_k))-\psi(\nabla_\sigma e_1,e_2,e_3,\eta_k)-\dots -\psi(e_1,e_2,e_3,\nabla_\sigma\eta_k)
$$
and equation \eqref{psi_associator}, and for the last one we used the skew-symmetry of $\nabla_\sigma\psi$ and the associativity condition  $\chi(e_1,e_2,e_3)=0$.
\end{proof}

\begin{corollary}
If $\varphi$ is torsion free (i.e. $\nabla\varphi=0$), then $a=A-A_0=0$. 
\end{corollary}
The purpose of this Section is to study in detail the expression for the
squared Fueter-Dirac operator obtained from \eqref{D_with_torsion}. Fix $p\in Y$ and choose  local orthonormal frames $\{e_1,e_2,e_3\}$ and $\{\eta_4,\eta_5,\eta_6,\eta_7\}$  of $TY$ and $NY$, respectively, such that 
        \begin{equation}\label{geodesic_frame}
        (\nabla_{e_i}{e_j})_p=(\nabla_{e_i}{\eta_k})_p=(\nabla_{\eta_l}{\eta_k})_p=0
        \end{equation}
        for all $i,j=1,2,3$ and $k,l=4,5,6,7$. Observe that, for any sections $\sigma,\eta\in \Omega^0(TM|_Y)$, one has
        \begin{equation}\label{splitting}
        \nabla_\sigma (\eta)\in \Omega^0(TM|_Y)=\Omega^0(TY)\oplus \Omega^0(NY),
        \end{equation}
        so both  tangent and normal components of \eqref{geodesic_frame} vanish at $p$. 
        Then the following holds  at $p$:
\begin{eqnarray*}\label{eq: D_A^2}
        \fueter^2\sigma 
        &=&\sum_{i,j=1}^3 e_i\times\n_i^\perp(e_j\times \n_j^\perp \sigma)-\sum_{j=1}^3\sum_{k=4}^7(\nabla_{e_j\times\n_j^\perp\sigma}\psi)(\eta_k,e_1,e_2,e_3)\eta_k\\
        &&-\sum_{i=1}^3\sum_{l=4}^7 e_i\times\n_i^\perp\{(\nabla_{\sigma}\psi)(\eta_l,e_1,e_2,e_3)\eta_l \} 
        +\sum_{k,l=4}^7(\nabla_{(\nabla_{\sigma}\psi)(\eta_l,e_1,e_2,e_3)\eta_l}\psi)(\eta_k,e_1,e_2,e_3)\eta_k\\
        &=&
        \underbrace{\sum_{i,j=1}^3 e_i\times(e_j\times\n_i^\perp\n_j^\perp\sigma)}_{\rm{(I)}}+\underbrace{\sum_{i,j,l=1}^3\sum_{m=4}^7T(e_i,e_l)\psi(e_l,e_j,\n_j^\perp\sigma,\eta_m)e_i\times \eta_m}_{\rm{(II)}}\\
&&-\underbrace{\sum_{j=1}^3\sum_{k,n=4}^7\varphi(e_j,\n_j^\perp\sigma,\eta_n)(\nabla_{\eta_n}\psi)(\eta_k,e_1,e_2,e_3)\eta_k}_{\rm{(III)}}
-\underbrace{\sum_{i=1}^3\sum_{l=4}^7 e_i(\nabla_{\sigma}\psi(\eta_l,e_1,e_2,e_3)) e_i\times \eta_l}_{\rm{(IV)}}\\
&&+\underbrace{\sum_{k,l=4}^7(\nabla_{\sigma}\psi)(\eta_l,e_1,e_2,e_3)(\nabla_{\eta_l}\psi)(\eta_k,e_1,e_2,e_3)\eta_k}_{\rm{(V)}}.
\end{eqnarray*}
To obtain  (I) and (II) we used Lemma \ref{Leibniz_rule} (i) and the property $(\nabla_ie_j)_p=0$, whereas (IV) follows from the Leibniz rule for $\nabla^\perp$ and $(\nabla_i\eta_k)_p=0$.

\begin{lemma}\label{lem: calculation of (I)}
Denoting by  $\nabla^\ast\nabla$  the Laplacian of the connection $\nperp$,
by $k$ the scalar curvature of $Y$, by $F^-$ the curvature associated to the spin bundle $S^-$, and by $\overline{\rho}$ the natural extension $\rho\otimes \mathbb{1}_{\End(S^-)}$ of $\rho:\Omega^2(Y)\rightarrow \End(S^+)$, one
has
$$
\mathrm{(I)} = \nabla^\ast\nabla\sigma+\frac{k}{4}\sigma + \overline{\rho}(F^-)\sigma.
$$
\end{lemma}
\begin{proof}
 In terms of an orthonormal frame   $\{e_1,e_2,e_3\}$ of $TY$,  
\begin{align*}
\mathrm{(I)} &=\sum_{i=1}^3 e_i\times(e_i\times\n_i^\perp\n_i^\perp\sigma )+\sum_{\substack{i,j=1\\ i\neq j}}^3 e_i\times(e_j\times\n_i^\perp\n_j^\perp\sigma )\\
    &=-\sum_i \n_i^\perp\n_i^\perp\sigma -\sum_{i\neq j} (e_i\times e_j)\times\n_i^\perp\n_j^\perp\sigma\\
    &=-\sum_i \n_i^\perp\n_i^\perp\sigma -\nabla_{\nabla^\top_ie_i}^\perp\sigma-\sum_{i<j} (e_i\times e_j)\times(\n_i^\perp\n_j^\perp-\n_j^\perp\n_i^\perp-\nabla_{[e_i,e_j]}^\perp)\sigma\\
    &=\nabla^\ast\nabla\sigma-\sum_{i<j} (e_i\times e_j)\times R^\perp(e_i,e_j)\sigma.
\end{align*}
Here $R^\perp\in \Omega^0(\Lambda^2 T^\ast Y\otimes \End(NY))$ is the normal curvature of $Y$:
        \begin{equation}\label{normal_curvature}
                  R^\perp(e_i,e_j)\sigma=(\n_i^\perp\n_j^\perp-\n_j^\perp\n_i^\perp-\nabla_{[e_i,e_j]}^\perp)\sigma.
        \end{equation}
        To obtain the second equality, we used  \eqref{def_associator} in each term of the form
\begin{align*}
  e_i\times(e_i\times\n_i^\perp\n_i^\perp\sigma )=&-\chi(e_i,e_i,\n_i^\perp\n_i^\perp\sigma )-\langle e_i,e_i\rangle\n_i^\perp\n_i^\perp\sigma+\langle e_i,\n_i^\perp\n_i^\perp\sigma\rangle e_i\\
  =&-\n_i^\perp\n_i^\perp\sigma.
\end{align*}
Moreover, for $i\neq j$,
\begin{align*}
e_i\times(e_j\times\n_i^\perp\n_j^\perp\sigma )=&-\chi(e_i,e_j,\n_i^\perp\n_j^\perp\sigma )-\langle e_i,e_j\rangle\n_i^\perp\n_j^\perp\sigma+\langle e_i,\n_i^\perp\n_j^\perp\sigma\rangle e_j\\
 =&-\chi(\n_i^\perp\n_j^\perp\sigma, e_i,e_j)=\n_i^\perp\n_j^\perp\sigma\times(e_i\times e_j)\\
 =&-(e_i\times e_j)\times\n_i^\perp\n_j^\perp\sigma.
\end{align*}

Since  $\nabla:=\nabla^{S^+\otimes S^-}=\nabla^+\otimes \mathbb{1}_{S^-} +\mathbb{1}_{S^+}\otimes \nabla^-$ agrees with the induced connection $\nperp$, one has  $R^\perp=F^\nabla =F^+\otimes \mathbb{1}_{S^-} +\mathbb{1}_{S^+}\otimes F^-$, where $F^\pm$ is the curvature of the connection $\nabla^\pm$. Now, using Proposition \ref{prop: NY=S+ x S-}, we identify the normal section $\sigma$ with the section $\kappa\otimes \varepsilon\in\Omega^0(S^+\otimes S^-)$, and recall that the Clifford product of the normal bundle $\gamma(e_i)\sigma=e_i\times\sigma$ coincides with Clifford multiplication
$$
 \tau:=\Gamma_0\otimes \mathbb{1}_{S^-}: TY\rightarrow \End(S^+\otimes S^-), 
$$
where $\Gamma_0$ is the spin structure on $TY$. Defining 
$$
 \mathcal{R}(\kappa\otimes \varepsilon) := \frac{1}{2}\sum_{ij} \tau(e_i)\tau(e_j)F_{ij}^\nabla(\kappa\otimes \varepsilon),
$$
and using \eqref{def_associator}, we have
\begin{align*}
-\sum_{i<j}(e_i\times e_j)\times R^\perp(e_i,e_j)\sigma=&\sum_{i<j}e_i\times( e_j\times R^\perp(e_i,e_j))\sigma\\
                                                 =&\frac{1}{2}\sum_{ij}\gamma(e_i)\gamma(e_j)(R^\perp(e_i,e_j)\sigma) \\
                                                 =&\frac{1}{2}\sum_{ij}\tau(e_i)\tau(e_j)F^\nabla_{ij}(\kappa\otimes \varepsilon)=\mathcal{R}(\kappa\otimes \varepsilon)
\end{align*}
Therefore,
\begin{align*}
\mathcal{R}(\kappa\otimes \varepsilon) &= \frac{1}{2}\sum_{ij} (\Gamma_0 \otimes \mathbb{1}_{S^-})(e_i)(\Gamma_0\otimes \mathbb{1}_{S^-})(e_j)F_{ij}^\nabla(\kappa\otimes \varepsilon) \\ 
                                   &= \frac{1}{2}\sum_{ij} (\Gamma_0\otimes \mathbb{1}_{S^-})(e_i)(\Gamma_0\otimes \mathbb{1}_{S^-})(e_j)(F_{ij}^+\kappa\otimes \varepsilon+\kappa\otimes F^-_{ij}\varepsilon) \\
                                   &= \underbrace{\frac{1}{2}\sum_{ij} (\Gamma_0(e_i)\Gamma_0(e_j)F_{ij}^+\kappa)\otimes \varepsilon}_{\mathrm{(I')}}+\underbrace{\frac{1}{2}\sum_{ij} (\Gamma_0(e_i)\Gamma_0(e_j)\kappa)\otimes F^-_{ij}\varepsilon.}_{\mathrm{(I'')}}                                  
\end{align*}
Each endomorphism $F^+_{ij}: \Omega^0(S^+)\rightarrow \Omega^0(S^+)$ is given by the formula (c.f. \cite[Theorem 4.15.]{lawson1989spin}) 
\begin{equation*}
F^+_{ij}\kappa=\frac{1}{2}\sum_{k<l}\langle R_{ij}(e_k),e_l\rangle \Gamma_0(e_k)\Gamma_0(e_l)\kappa,
\end{equation*}
where $R_{ij}=R(e_i,e_j)$  is the Riemann tensor of the induced connection on $Y$, with components  $R_{ijk}^l=\langle R_{ij}(e_k),e_l\rangle$. Then, for the first term,
\begin{eqnarray*}
\mathrm{(I')}&=& \frac{1}{8}\sum_{ijkl} R_{ijk}^l\Gamma_0(e_i)\Gamma_0(e_j)\Gamma_0(e_k)\Gamma_0(e_l)\kappa \\
    &=& \frac{1}{8}\sum_l \biggl(\sum_{ij,(i=k)} R_{iji}^l \Gamma_0(e_i)\Gamma_0(e_j)\Gamma_0(e_i)+\sum_{ij,(j=k)} R_{ijj}^l    \Gamma_0(e_i)\Gamma_0(e_j)\Gamma_0(e_j)\\  
    &&+\frac{1}{3}\sum_{i\neq j\neq k\neq i} (R_{ijk}^l+R_{jki}^l+R_{kij}^l)\Gamma_0(e_i)\Gamma_0(e_j)\Gamma_0(e_k)
    \biggr)\Gamma_0(e_l)\kappa\\
    &=&\frac{1}{8}\sum_l \biggl( \sum_{ij} R_{iji}^l \Gamma_0(e_j)-\sum_{ij} R_{ijj}^l\Gamma_0(e_i)\biggr)\Gamma_0(e_l)\kappa \\
    &=&\frac{1}{4}\sum_{ijl} R_{iji}^l \Gamma_0(e_j)\Gamma_0(e_l)\kappa =\frac{1}{8}\sum_{ijl} R_{iji}^l (\Gamma_0(e_j)\Gamma_0(e_l)+\Gamma_0(e_l)\Gamma_0(e_j))\kappa \\
    &=&-\frac{1}{4}\sum_{ijl} R_{iji}^l \delta_j^l\kappa=-\frac{1}{4}\sum_{ij} R_{iji}^j\kappa =\frac{1}{4}k\kappa, 
\end{eqnarray*}
where $k$ denotes the scalar curvature of $Y$. For the second term, recall that $\Omega ^2(Y,\End(S^-))\cong \Omega ^2(Y)\otimes \Omega ^0(\End(S^-))$, so 
$$
F^-=\sum_{i<j}(e_i\wedge e_j) \otimes F^-_{ij}.
$$
Moreover, observe that $\Gamma_0$ (also $\gamma$) induces a map 
$\rho: \Omega^2 (Y) \rightarrow \End(S^+)$ defined
by\begin{equation*}
\rho(\sum_{i<j}\eta_{ij}e_i\wedge e_j):=\sum_{i<j} \eta_{ij} \Gamma_0(e_i)\Gamma_0(e_j)
\end{equation*} 
and consider the extension
\begin{equation}\label{representaçao}
\overline{\rho}:=\rho \otimes \mathbb{1}_{\End(S^-)}: \Omega^2(Y,\End(S^-))  \rightarrow \End(S^+\otimes S^-)
\end{equation}
given  by 
$$
\overline{\rho}(\sum_{i<j}(e_i\wedge e_j)\otimes F^-_{ij}):=\sum_{i<j} (\Gamma_0(e_i)\Gamma_0(e_j)\otimes F^-_{ij}).
$$
Then, 
\begin{align*}
\mathrm{(I'')}&=\frac{1}{2}\sum_{ij} (\Gamma_0(e_i)\Gamma_0(e_j)\otimes F_{ij}^-)(\kappa\otimes \varepsilon)\\
   &=\frac{1}{2}\overline{\rho}(\sum_{ij}(e_i\wedge e_j)\otimes F_{ij}^-)(\kappa\otimes \varepsilon)\\
   &=\overline{\rho}(F^-)(\kappa\otimes \varepsilon).\qedhere
\end{align*}
\end{proof}

\subsection{First-order corrections}\label{1st_corrections}

The correction terms (II),...,(V) can be conveniently organised into three \nth{1} order differential operators $P_1,P_2,P_3$ on sections of $NY$. 
\begin{lemma} \label{lem: calculation of (II)}
$$
  \mathrm{(II)}=P_1(\sigma):=\sum_{i,j=1}^3 T_{ii} e_j\times\n_j^\perp\sigma-T_{ji} e_j\times\n_i^\perp\sigma-2\sum_{(i,j,k)\in S_3^0}^3C_{ij}\n_k^\perp\sigma,
$$
where $S_3^0$ are the even permutations in $S_3$,  
$T_{ji}$ is the full torsion tensor and $C_{ij}$ the anti-symmetric part of $T_{ij}$.
\begin{proof}
By Lemma \ref{Omega3}, we have 
$$
\mathrm{(II)}=\sum_{i,j,n=1}^3\sum_{k=4}^7T(e_i,e_n)\psi(e_n,e_j,\n_j^\perp\sigma,\eta_k)e_i\times \eta_k=(\ast).
$$
Since $\chi(e_n,e_j,\n_j^\perp\sigma)\in \Omega^0(NY)$, then using \eqref{def_associator} we have
\begin{align*}
(\ast)=&\sum_{i,j,n=1}^3 T(e_i,e_n)e_i\times\chi(\n_j^\perp\sigma,e_n,e_j)=\sum_{i,j,n=1}^3 -T(e_i,e_n)e_i\times(\n_j^\perp\sigma\times(e_n\times e_j))\\
      =&\sum_{i,j,n=1}^3 T(e_i,e_n)\chi(e_i,\n_j^\perp\sigma, e_n\times e_j)-\langle e_i,e_n\times e_j\rangle\n_j^\perp\sigma\\
      =&\sum_{i,j,n=1}^3 T(e_i,e_n)(\n_j^\perp\sigma\times(e_i\times (e_n\times e_j))-\varphi(e_i,e_n,e_j)\n_j^\perp\sigma)
\end{align*}
Using  relations $e_1\times e_2=e_3$ and $e_i\times(e_n\times e_j)=-\chi(e_i,e_n,e_j)-\langle e_i,e_n\rangle e_j+\langle e_i,e_j\rangle e_n$. The first term of the sum is equal to 
$$
\sum_{i,j=1}^3 T_{ii} e_j\times\n_j^\perp\sigma-T_{ji} e_j\times\n_i^\perp\sigma.
$$
Moreover, since  $\varphi(e_1,e_2,e_3)=1$, the second term becomes
\begin{equation}
-2\sum_{(i,j,k)\in S_3^0}^3C_{ij}\n_k^\perp\sigma. 
\end{equation}
where $2C_{ij}=T_{ij}-T_{ji}$.
\end{proof}
\end{lemma}

\begin{lemma}   \label{torsion_term}
With the above notation 
\begin{equation}
\sum_{k=4}^7 (\nabla_n\psi)(\eta_k,e_1,e_2,e_3)\eta_k=-\sum_{k=4}^7T_{nk}\eta_k.     
\end{equation}
\end{lemma}

\begin{proof}
Since $Y$ is associative, Corollary \ref{covariant_psi} gives $\nabla_n\psi_{k123}=-T_{nk}$.
\end{proof}

Denote the following two operators on $NY$, involving the full torsion tensor
$$
P_2(\sigma)=\sum_{i=1}^3\sum_{l=4}^7((\nabla_iT)(\sigma,\eta_l)+T(\n_i^\perp\sigma, \eta_l))e_i\times \eta_l,
$$
$$
P_3(\sigma)=\sum_{k,l=4}^7\biggl(T(\sigma,\eta_l)+\sum_{i=1}^3\varphi(e_i,\n_i^\perp\sigma,\eta_l)\biggr)T_{lk}\eta_k.
$$
With this notation, we arrive at one of our main theorems:
\begin{theorem}
The Weitzenböck formula for \eqref{fueter_dirac_operator} is
\begin{align}\label{fueter_dirac_formula}
\fueter^2(\sigma)&=\nabla^\ast\nabla\sigma+\frac{1}{4}k\cdot\sigma+\overline{\rho}(F^-)\sigma+P_1(\sigma)+P_2(\sigma)+P_3(\sigma)
\end{align}

\begin{proof}
We examine the five components
of $\fueter^2$ as on page \pageref{eq: D_A^2}. Components (I) and (II) have been studied in Lemmata \ref{lem: calculation of (I)} and \ref{lem: calculation of (II)}. Now, applying Lemma \ref{torsion_term}, we have 
$$\mathrm{(III)}=\sum_{i=1}^3\sum_{k,l=4}^7\varphi(e_i,\n_i^\perp\sigma,\eta_l)T_{lk}\eta_k.$$

As to (IV), for each $i=1,2,3$ and $l=4,5,6,7$, we use Lemma \ref{torsion_term} to find
\begin{align*}
e_i((\nabla_\sigma\psi)(\eta_l,e_1,e_2,e_3))=-e_i(T(\sigma,\eta_l))=-(\nabla_iT)(\sigma,\eta_l)-T(\n_i^\perp\sigma, \eta_l).
\end{align*}
Then, indeed,
$$\mathrm{(IV)}=\sum_{i=1}^3\sum_{l=4}^7((\nabla_iT)(\sigma,\eta_l)+T(\n_i^\perp\sigma, \eta_l))e_i\times \eta_l=P_2(\sigma).$$

Finally, a simple calculation gives
$\displaystyle
\mathrm{(V)}=\sum_{k,l=4}^7T(\sigma,\eta_l)T_{lk}\eta_k$,
and
\begin{equation*}
\mathrm{(V)}+\mathrm{(III)}=\sum_{k,l=4}^7\biggl(T(\sigma,\eta_l)+\sum_{i=1}^3\varphi(e_i,\n_i^\perp\sigma,\eta_l)\biggr)T_{lk}\eta_k=P_3(\sigma)
\qedhere
\end{equation*}
\end{proof}
\end{theorem}

\begin{corollary}
Let $(M^7,\varphi)$ be a $\gt$-manifold. Then,
$$
\fueter^2=\dirac^2=\nabla^\ast \nabla +\frac{1}{4}k+\overline{\rho}(F^-)
$$
\end{corollary}



In \cite{gayet2014smooth}, Gayet obtains a Weitzenböck-type formula when the $\gt$-structure is torsion-free:
\begin{equation}\label{Gayet_formula}
\dirac^2=\nabla^\ast \nabla+\mathcal{R}-\mathcal{A}.
\end{equation}
The term $\mathcal{R}(\sigma)=\pi^\perp\sum_{i=1}^3 R(e_i,\sigma)e_i$ can be seen as a partial Ricci operator, where $R$ is the curvature tensor of $g$ on $M$ and $\pi^\perp$ is the orthogonal projection to $NY$, and  
$$
\mathcal{A}: \Omega^0(NY)\rightarrow \Omega^0(\Sym(TY)),
$$
defined by $\mathcal{A}(\sigma)=S^t\circ S(\sigma)$, is a symmetric positive \nth{0}--order operator determined by the shape operator $S(\sigma)(X)=-(\nabla_X\sigma)^\top$. With these data, Gayet formulates a vanishing theorem for a compact associative submanifold $Y$ of a $\gt$--manifold and proves that $Y$ is rigid when the spectrum of the operator $\mathcal{R}-\mathcal{A}$ is positive. The advantage of  formula \eqref{Gayet_formula} lies in the relation between the intrinsic and extrinsic geometries of the associative submanifold, because $\mathcal{R}-\mathcal{A}$ is obtained from a curvature term 
\begin{equation}        \label{eq: curvature term}
        -\sum_{i<j}^3(e_i\times e_j)\times R^\perp(e_i,e_j)\sigma.
\end{equation}
While one cannot entirely apply his proof to the general case (because the full torsion tensor is nonzero), we are able to adapt some of its steps.

Given  $u,v\in \Omega^0(TY)$, and their respective local extensions $\bar{u},\bar{v}$ to $TM$, the $TY$-valued $2$--form   $B(u,v):=\nabla_{\bar{u}}\bar{v}-\nabla^\top_u v$ relates to the shape operator by $\langle B(u,v),\sigma\rangle=\langle S_\sigma(u),v\rangle$, for $\sigma\in\Omega^0(TM)$. These data define a natural $\nth{0}$--order operator
\begin{equation}\label{operator_B}
\mathcal{B}(\sigma):=\sum_{i,j=1}^3(e_i\times e_j)\times B(e_j,S_{\sigma}(e_i)).
\end{equation}

\begin{prop}\label{alternative_expression}
The curvature term in (\ref{eq: curvature term}) can be rewritten as
\begin{equation}
      -\sum_{i<j}^3(e_i\times e_j)\times R^\perp(e_i,e_j)\sigma
      =
      \mathcal{R}(\sigma)+\mathcal{B}(\sigma)-\pi^\perp\biggl(\sum_{i\in\mathbb{Z}_3}e_i\times\mathcal{T}(e_{i+1},\sigma,e_i,e_{i+1})\biggr),
\end{equation}
where $\mathcal{R}(\sigma)=\pi^\perp\sum_{i=1}^3 R(e_i,\sigma)e_i$ is the partial Ricci operator defined in \eqref{Gayet_formula} and $\mathcal{T}$ is defined in \eqref{torsion_four_tensor} by
\begin{align*}
            \mathcal{T}(e_{i+1},\sigma,e_i,e_{i+1})
            &:=\sum_{m=1}^7T(\sigma,e_m)(\nabla_{i+1}\psi)(e_m,e_i,e_{i+1},\cdot)^\sharp-T_{i+1 m}             (\nabla_\sigma\psi)(e_m,e_i,e_{i+1},\cdot)^\sharp\\
            &\qquad+\Bigl((\nabla_{i+1}T)(\sigma,e_m)-(\nabla_\sigma T)(e_{i+1},e_m)\Bigr)\chi(e_m,e_i,e_{i+1}).
\end{align*}
\end{prop}

\begin{proof} Expanding the summands in the frame $\{\eta_4,\dots,\eta_7\}$ and using
anti-symmetry of the mixed product and the Ricci equation, we have
\begin{align*}
        -\sum_{i<j}^3
        (e_i\times e_j)\times R^\perp(e_i,e_j)\sigma 
        &=
        -\frac{1}{2}\sum_{i,j=1}^3\sum_{k=4}^7
        \langle (e_i\times e_j)\times R^\perp(e_i,e_j)\sigma, \eta_k\rangle \eta_k\\
        &=\frac{1}{2}\sum_{i,j=1}^3\sum_{k=4}^7
        \langle R^\perp(e_i,e_j)\sigma,(e_i\times e_j)\times  \eta_k\rangle \eta_k\\
        &=
        \frac{1}{2}\sum_{i,j=1}^3\sum_{k=4}^7\langle R(e_i,e_j)\sigma,(e_i\times e_j)\times \eta_k\rangle \eta_k+\langle [S_\sigma, S_{(e_i\times e_j)\times \eta_k}]e_i,e_j\rangle \eta_k\\
        &=
        \underbrace{
                -\frac{1}{2}\pi^\perp{\sum_{i,j=1}^3} 
                (e_i\times e_j)\times R(e_i,e_j)\sigma
                }_{(\star)} 
        +
        \underbrace{
                \frac{1}{2}\sum_{i,j=1}^3\sum_{k=4}^7
                \langle [S_\sigma, S_{(e_i\times e_j)\times \eta_k}]e_i,e_j\rangle
                 \eta_k
                 }_{(\star\star)}.
\end{align*}
Applying the Bianchi identity $R(e_i,e_j)\sigma=-R(\sigma, e_i)e_j-R(e_j,\sigma)e_i$
to the first term, expanding the sum and using  Lemma \ref{Leibniz_rule}, we have:
\begin{align*}
        (\star)
        &=
        \pi^\perp\sum_{i,j=1}^3(e_i\times e_j)\times R(e_j,\sigma)e_i\\
        &=
        \pi^\perp (e_3\times R(e_2,\sigma)e_1-e_2\times R(e_3,\sigma)e_1-e_3\times R(e_1,\sigma)e_2+e_1\times R(e_3,\sigma)e_2\\
        &\quad + e_2\times R(e_1,\sigma)e_3-e_1\times R(e_2,\sigma)e_3)\\
        &=\pi^\perp
        (
        \underbrace{-e_1\times [R(e_2,\sigma)e_1\times e_2+e_1\times R(e_2,\sigma)e_2}_{\mathrm{(I)}}+
        \mathcal{T}  
        (e_2,\sigma,e_1,e_2)]\\
        &\quad\underbrace{-e_2\times [R(e_3,\sigma)e_2\times e_3+e_2\times R(e_3,\sigma)e_3}_{\mathrm{(II)}}+
        \mathcal{T}
        (e_3,\sigma,e_2,e_3)]\\
        &\quad-\underbrace{e_3\times [R(e_1,\sigma)e_3\times e_1+e_3\times R(e_1,\sigma)e_1}_{\mathrm{(III)}}+
        \mathcal{T}
        (e_1,\sigma,e_3,e_1)]\\
        &\quad+e_3\times R(e_2,\sigma)e_1+e_1\times R(e_3,\sigma)e_2+e_2\times R(e_1,\sigma)e_3
        ).
\end{align*}
Using the identity $u\times(v\times w)+v\times(u\times w)=\langle u,w\rangle v+\langle v,w\rangle u-2\langle u,v\rangle w$, we check that 
\begin{align*}
\mathrm{(I)}&=-e_3\times R(e_2,\sigma)e_1-(e_2,\sigma,e_1,e_2)e_1+2(e_2,\sigma,e_1,e_1)e_2+R(e_2,\sigma)e_2\\
\mathrm{(II)}&=-e_1\times R(e_3,\sigma)e_2-(e_3,\sigma,e_2,e_3)e_2+2(e_3,\sigma,e_2,e_2)e_3+R(e_3,\sigma)e_3\\
\mathrm{(III)}&=-e_2\times R(e_1,\sigma)e_3-(e_1,\sigma,e_3,e_1)e_3+2(e_1,\sigma,e_3,e_3)e_1+R(e_1,\sigma)e_1,
\end{align*}
where $(e_1,\sigma,e_3,e_1):=\langle R(e_1,\sigma)e_3,e_1\rangle$. Cancelling
terms and taking the orthogonal projection on $\mathrm{(I)+(II)+(III)}$, we find $(\star)=\mathcal{R}(\sigma)-\pi^\perp\Bigl(\sum e_i\times\mathcal{T}(e_{i+1},\sigma,e_i,e_{i+1})\Bigr)$.

Finally, by the symmetry of $S_\sigma$ and $S_{(e_i\times e_j)\times \eta_k}$, the second term is
\begin{align*}
        (\star\star)
        &=\frac{1}{2}\sum_{i,j=1}^3\sum_{k=4}^7
        \biggl(\langle S_{(e_i\times e_j)\times \eta_k}(e_i),S_\sigma(e_j)\rangle
        -
        \langle S_\sigma(e_i),S_{(e_i\times e_j)\times \eta_k} (e_j)\rangle\biggr)\eta_k\\
        &=\sum_{i,j=1}^3\sum_{k=4}^7
        \biggl(\langle S_{(e_i\times e_j)\times \eta_k}(e_i),S_\sigma(e_j)\rangle\biggr)\eta_k\\
        &=\sum_{i,j=1}^3\sum_{k=4}^7
        \biggl(\langle B(e_i,S_\sigma(e_j)),(e_i\times e_j)\times \eta_k\rangle\biggr)\eta_k\\
        &=\sum_{i,j=1}^3\sum_{k=4}^7
        \varphi(e_i\times e_j ,\eta_k, B(e_i,S_\sigma(e_j)))\eta_k\\
        &=-\sum_{i,j=1}^3(e_i\times e_j)\times B(e_i,S_\sigma(e_j))\\
        &=\mathcal{B}(\sigma). \qedhere
\end{align*}
\end{proof}

\section{The nearly parallel case and applications}
\label{sec: nearly parallel case}

The torsion-free condition for a   $\gt$-structure is highly overdetermined, so examples are difficult to construct and seldom known explicitly. In terms of the Fern\'andez-Gray classification recalled in Section \ref{sec: torsion tensor}, the next natural  `least-torsion' case consists of the so-called nearly parallel structures, for which the torsion forms $\tau_1,\tau_2,\tau_3$  vanish and the remaining torsion is just a constant:
\begin{definition}
Let $(M,\varphi)$ a manifold with a $\gt$--structure, $\varphi$ is called \emph{nearly parallel} if 
$$d\varphi=\tau_0\psi,$$with $\tau_0\neq 0$ constant. 
\end{definition}

Regarding the deformations of associative submanifolds, our approach unifies previously known results by means of a Bochner-type vanishing theorem. This technique requires a certain `positivity' of curvature, which can in practice be found in cases of interest studied by several authors. 

\subsection{Proof of the vanishing theorem}
Following Proposition \ref{prop: full torsion tensor}, the full torsion tensor in the nearly parallel case is given by $T_{ij}=\frac{\tau_0}{4}g_{ij}$,
which drastically simplifies the Weitzenb\"ock formula \eqref{fueter_dirac_formula}:

\begin{prop}
The Weitzenböck formula for the Fueter-Dirac operator \eqref{fueter_dirac_operator} in the nearly parallel case is
\begin{equation}\label{nearly_parallel_case}
\fueter^2(\sigma)=\nabla^\ast\nabla\sigma+\frac{1}{4}k\cdot \sigma+\rho(F^-)\sigma+\tau_0\dirac(\sigma)+\frac{\tau_0^2}{16}\cdot\sigma.
\end{equation}
\end{prop}

\begin{proof}
Given the orthonormal frame $\{e_1,e_2,e_3,\eta_4,...,\eta_7\}$, it suffices to prove that the last three terms in \eqref{fueter_dirac_formula}  satisfy $$
(P_1+P_2+P_3)(\sigma)
=
\tau_0\dirac(\sigma)+\frac{\tau_0^2}{16}\cdot\sigma.
$$
At a point $p\in Y$, for $P_1$, we have $C_{ij}=0$, because $\tau_1$ and $\tau_2$ are zero, then
\begin{align*}
\sum_{i,j=1}^3 T_{ii} e_j\times\n_j^\perp\sigma-T_{ji} e_j\times\n_i^\perp\sigma &=\frac{3}{4}\tau_0\sum_{j=1}^3e_j\times\n_j^\perp\sigma-\frac{1}{4}\tau_0\sum_{j=1}^3e_j\times\n_j^\perp\sigma\\
&=\frac{1}{2}\tau_0\dirac(\sigma).
\end{align*}
For $P_2$,
\begin{align*}
\sum_{i=1}^3\sum_{l=4}^7((\nabla_iT)(\sigma,\eta_l)+T(\n_i^\perp\sigma, \eta_l))e_i\times \eta_l&=\frac{\tau_0}{4}\sum_{i=1}^3\sum_{l=4}^7g(\n_i^\perp\sigma,\eta_l)e_i\times \eta_l\\
&=\frac{\tau_0}{4}\sum_{i=1}^3e_i\times\n_i^\perp\sigma=\frac{\tau_0}{4}\dirac(\sigma).
\end{align*}
And, for $P_3$,
\begin{align*}
\sum_{k,l=4}^7\biggl(T(\sigma,\eta_l)+\sum_{i=1}^3\varphi(e_i,\n_i^\perp\sigma,\eta_l)\biggr)T_{lk}\eta_k &=\frac{\tau_0}{4}\sum_{k,l=4}^7\biggl(\frac{\tau_0}{4}g(\sigma,\eta_l)+\sum_{i=1}^3\varphi(e_i,\nabla_i\sigma,\eta_l)\biggr)g(\eta_l,\eta_k)\eta_k\\
&=\frac{\tau_0}{4}\sum_{l=4}^7\biggl(\frac{\tau_0}{4}g(\sigma,\eta_l)+\sum_{i=1}^3\varphi(e_i,\nabla_i\sigma,\eta_l)\biggr)\eta_l\\
&=\frac{\tau_0^2}{16}\cdot\sigma +\frac{\tau_0}{4}\dirac(\sigma).
\qedhere
\end{align*}
\end{proof} 

\begin{corollary}\label{alternative_weitzenbock_formula}
Equation \eqref{nearly_parallel_case} can be rewritten as
\begin{equation}\label{nearly-weitzenbock}
        \fueter^2(\sigma)
        =
        \nabla^\ast\nabla\sigma+\mathcal{R}(\sigma)+\mathcal{B}(\sigma)+\tau_0\dirac(\sigma)+\frac{\tau_0^2}{4}\cdot\sigma.
\end{equation}
\begin{proof}
For a nearly parallel $\gt$-structure, the full torsion tensor is $T_{ij}=\frac{\tau_0}{4}g_{ij}$, thus $\nabla T=0$ and so:
\begin{align*}
        \sum_{i\in\mathbb{Z}_3}e_i\times\mathcal{T}(e_{i+1},\sigma,e_i,e_{i+1})=&\frac{\tau_0}{4}\sum_{i\in\mathbb{Z}_3}\sum_{m,l=1}^7\Big(g(\sigma,e_m)(\nabla_{i+1}\psi)(e_m,e_i,e_{i+1},e_l)\\
        &-g(e_{i+1},e_m)(\nabla_\sigma\psi)(e_m,e_i,e_{i+1},e_l)\Big)e_i\times e_l\\
        =&\frac{\tau_0}{4}\sum_{i\in\mathbb{Z}_3}\sum_{l=1}^7\Big((\nabla_{i+1}\psi)(\sigma,e_i,e_{i+1},e_l)\\
        &-(\nabla_\sigma\psi)(e_{i+1},e_i,e_{i+1},e_l)\Big)e_i\times e_l\\
         =&\frac{\tau_0}{4}\sum_{i\in\mathbb{Z}_3}\sum_{l=1}^7\Big((\nabla_{i+1}\psi)(\sigma,e_i,e_{i+1},e_l)\Big)e_i\times e_l\\
         =&-\frac{\tau_0}{4}\sum_{i\in\mathbb{Z}_3}\sum_{l=1}^7T(e_{i+1},e_{i+1})\varphi(\sigma,e_i,e_l)e_i\times e_l\\
         =&-\frac{\tau_0^2}{16}\sum_{i\in\mathbb{Z}_3}g(e_{i+1},e_{i+1})e_i\times(\sigma\times e_i)\\
        =&-\frac{3}{16}\tau_0^2\sigma
\end{align*}
Here we used the skew-symmetry of $\n_\sigma\psi$ for the third equality and Corollary \ref{covariant_psi} for the fourth one. Equation \eqref{nearly-weitzenbock} now follows from Proposition \ref{alternative_expression}.
\end{proof}
\end{corollary}

\begin{theorem}\label{rigidity_theorem}
Let $(M,\varphi)$ be a $7$-manifold with a nearly parallel $\gt$--structure. If  $Y\subset M$ is a closed associative submanifold such that the operator $\bar{\rho}(F^-)$ associated to the curvature of the bundle $S^-$ (c.f. (\ref{eq: rho barra}) and Lemma \ref{lem: calculation of (I)}) is bounded below by $-(\frac{k}{4}-\frac{3}{16}\tau_0^2)$, then $Y$ is rigid.  
\begin{proof}
Let $\sigma$ be a section of $NY$, 
\begin{align*}
\Delta|\sigma|^2 &=\sum_i e_ie_i\langle \sigma, \sigma\rangle= 2\sum_i e_i\langle \n_i^\perp \sigma, \sigma\rangle\\
             &=2\sum_i \langle \n_i^\perp \n_i^\perp \sigma, \sigma\rangle +\langle \n_i^\perp \sigma,\n_i^\perp \sigma\rangle\\
             &=-2\langle \nabla^\ast \nabla \sigma,\sigma\rangle + 2|\nperp \sigma|^2\\
             &=-2\langle \fueter^2(\sigma),\sigma\rangle+\frac{k}{2}\langle \sigma,\sigma\rangle+2\langle\overline{\rho}(F^-)\sigma,\sigma\rangle+2\tau_0\langle\dirac(\sigma),\sigma\rangle +\frac{\tau_0^2}{8}|\sigma|^2+2|\nperp \sigma|^2.
\end{align*}
Taking $\sigma\in \ker \fueter$, equation \eqref{D_with_torsion} gives 
 $$\langle \dirac(\sigma),\sigma\rangle =\sum_{k=4}^7(\nabla_{\sigma}\psi)(\eta_k,e_1,e_2,e_3)\langle \eta_k,\sigma\rangle=-\sum_{k=4}^7T(\sigma ,\eta_k)\langle \eta_k,\sigma\rangle=-\frac{\tau_0}{4}\sum_{k=4}^7\langle\sigma ,\eta_k\rangle^2.$$
By Stokes' theorem, it follows
that\begin{align*}
0=&\int_Y \biggl(\frac{k}{4}|\sigma|^2+\langle \overline{\rho}(F^-)\sigma,\sigma\rangle -\frac{\tau_0^2}{4}\sum_{k=4}^7\langle\sigma,\eta_k\rangle^2+\frac{\tau_0^2}{16}|\sigma|^2+|\nperp \sigma|^2\biggr)d\vol_Y\\
 =&\int_Y (\biggl(\frac{k}{4}-\frac{3}{16}\tau_0^2\biggr)|\sigma|^2+\langle \overline{\rho}(F^-)\sigma,\sigma\rangle+|\nperp\sigma|^2)d\vol_Y.
\end{align*}
By assumption, $\langle \bar{\rho}(F^-)\sigma,\sigma\rangle\geq-\biggl(\frac{k}{4}-\frac{3}{16}\tau_0^2\biggr)\langle \sigma,\sigma\rangle$, so $\nperp\sigma=0$ and this implies $\dirac(\sigma)=0$. Notice  from Lemma \ref{torsion_term} that the Fueter-Dirac operator is
$$
\fueter = \dirac +\frac{\tau_0}{4} \qwithq \tau_0\neq 0.
$$
Then, from $\fueter(\sigma)=0$ it follows that $\sigma=0$, i.e. $\ker \fueter=\{0\}$.
\end{proof}
\end{theorem}


\subsection{The associative submanifolds of the $7$-sphere}

In \cite{lotay2012associative}, Lotay defines a $\gt$--structure $\varphi$ on $S^7$, writing $\mathbb{R}^8\setminus\{0\} \cong \mathbb{R}^+\times S^7$, such that 
$$\Phi_0|_{(r,p)}=r^3dr\wedge\varphi|_p+r^4\ast\varphi|_p,$$
where $\Phi_0$ is the $\spin(7)$--structure of $\mathbb{R}^8$, $r$ the radial coordinate on $\mathbb{R}^+$ and $\ast$ the Hodge star on $S^7$ induced by the round metric. Since $\Phi_0$ is closed, it follows that $d\varphi=4\ast\varphi$ i.e. $\varphi$ is a nearly parallel $\gt$--structure.

Consider the $7$--sphere as the homogeneous space $\spin(7)/G_2$, viewing $\spin(7)$ as the $\gt$ frame bundle over $S^7$. From the 
structure equations of $\spin(7)$ \cite[Chapter 4]{lotay2012associative},  the \emph{second fundamental form} $B\in \Omega^0(\Sym^2(TY)^\ast\otimes NY)$ (c.f. \cite[Definition 4.5]{lotay2012associative}) satisfies
\begin{equation}\label{second_fundamental_form_relation}
\sum_{i=1}^3e_i\times B(e_i,e_j)=0.
\end{equation}
Using \eqref{second_fundamental_form_relation}, the operator $\mathcal{B}(\sigma)$ from \eqref{operator_B} is given by
\begin{align*}
\mathcal{B}(\sigma)&=-\sum_{i=1}^3e_i\times\biggl(\sum_{j=1}^3e_j\times B(e_j,S_\sigma(e_i)\biggr)-\sum_{i,j=1}^3\langle e_i,e_j\rangle B(e_j,S_\sigma(e_i))\\
&=-\sum_{i=1}^3B(e_i, S_\sigma(e_i))
\end{align*}
Taking the inner product with the section $\sigma$ itself,
one obtains the non-positivity property\begin{align*}
\langle \mathcal{B}(\sigma),\sigma\rangle&=-\sum_{i=1}^3\langle B(e_i,S_\sigma(e_i)),\sigma\rangle=-\sum_{i=1}^3\langle S_\sigma(e_i),S_\sigma(e_i)\rangle=-\sum_{i=1}^3\left\Vert S_\sigma(e_i)\right\Vert^2.
\end{align*}

Consider the action of SU$(2)$ on $S^7$ given by
\begin{equation}\label{SU_action}
\begin{pmatrix}
z_1 \\ z_2 \\ z_3 \\ z_4 
\end{pmatrix} \mapsto \begin{pmatrix}
az_1+bz_2 \\ -\bar{b}z_1+\bar{a}z_2 \\ az_3+bz_4 \\ -\bar{b}z_3+\bar{a}z_4
\end{pmatrix} \ \ \ \text{for} \ \ \begin{pmatrix}
a        &    b   \\
-\bar{b} & \bar{a}
\end{pmatrix}\in \SU(2).
\end{equation}
By Corollary \ref{alternative_weitzenbock_formula}, we have,
$$
\fueter^2(\sigma)=\nabla^\ast\nabla\sigma+\mathcal{R}(\sigma)+\mathcal{B}(\sigma)+4\fueter(\sigma),
$$
or, in terms of the operator $\dirac$, \begin{equation}\label{Kawai_formula}
\dirac^2=\nabla^\ast\nabla\sigma+\mathcal{R}(\sigma)+\mathcal{B}(\sigma)+2\dirac(\sigma)+3\sigma,
\end{equation}
which coincides with the formula given by Kawai \cite{kawai2014deformations}. Consider the orbit $S^3\subset S^7$ of the  point $(1,0,0,0)$ under action \eqref{SU_action}. The tangent space to $S^3$ at a point $(z_1,z_2,0,0)$ is spanned by the vectors
$$
  X_1=(z_2,-z_1,0,0), \ \ X_2=(iz_2,iz_1,0,0), \ \  X_3=(iz_1,-iz_2,0,0).
$$ 
As the induced metric on $S^3$, from the round metric on $S^7$, coincides with the round metric of constant curvature $1$, the following results of \cite{bar1996dirac} can be adapted to our case.

\begin{lemma}\label{trivialized_normal_bundle}
The normal bundle $NS^3$ can be trivialized by parallel sections $\sigma_1,\dots,\sigma_4$ of the connection $\nperp$.
\begin{proof}
It suffices to show that the curvature operator $R^\perp$ vanishes (c.f. \eqref{normal_curvature}). Let $u,v$ be tangent vector fields of $S^3$, and $\sigma$ a section of $NS^3$, then the Ricci equation gives
\begin{align*}
    R^\perp(u,v)\sigma=&\sum_{k=4}^7\langle R^\perp(u,v)\sigma,\eta_k\rangle\eta_k\\
                      =&\sum_{k=4}^7(\langle R(u,v)\sigma,\eta_k\rangle + \langle [S_\sigma,S_{\eta_k}]u,v\rangle)\eta_k\\
                      =&\sum_{k=4}^7(\langle u,\sigma\rangle \langle v,\eta_k\rangle-\langle v,\sigma\rangle\langle u,\eta_k\rangle)\eta_k=0.
\end{align*}
At the third equality we used the well-known facts that the metric on $S^7$ has constant sectional curvature equal to $1$ and that $S^3 \subset S^7$ is a totally geodesic immersed submanifold.
\end{proof}
\end{lemma}

The following Weitzenböck formula relates the operator $D=\dirac-\ident$  with the Laplacian of the connection $\nperp$ on $NS^3$.

\begin{lemma}
On the normal bundle $NS^3$, the following formula holds:
\begin{equation}\label{Dirac_tilde}
D^2=\nabla^\ast\nabla+\ident.
\end{equation}
\end{lemma}

\begin{proof}
In a local orthonormal frame  $e_1,e_2,e_3$ around $p\in S^3$, we compute
\begin{align*}
D^2(\sigma) 
        &=\dirac^2(\sigma)-2\dirac(\sigma)+\sigma\\
        &=\nabla^\ast\nabla\sigma+\mathcal{R}(\sigma)+4\sigma\\
        &=\nabla^\ast\nabla\sigma+\Bigl(\sum_{i=1}^3\langle\sigma, e_i\rangle e_i-\langle e_i,e_i\rangle \sigma\Bigr)^\perp+4\sigma\\
        &=\nabla^\ast\nabla\sigma+\sigma.
        \qedhere
\end{align*}
\end{proof}

Consider a  basis $1= f_0, f_1,f_2,\dots$ of $L^2(S^3,\mathbb{R})$, consisting of eigenfunctions of the Laplace operator: 
$$
\Delta f_i=\lambda_if_i.
$$ 
The next lemma describes a natural eigenbasis for the operator $D^2$ on sections of  $NS^3$.

\begin{lemma}\label{eigenvalue_Dirac_square}
$D^2(f_i\sigma_k)=(\lambda_i+1)(f_i\sigma_k)$.
\begin{proof}
This follows directly from Lemma \ref{trivialized_normal_bundle} and 
 \eqref{Dirac_tilde}.
\end{proof}
\end{lemma}

Since the metric on $S^3$ has constant curvature 1, the eigenvalues of the Laplace operator on $S^3$ are 
$$
  \lambda_k=k(k+2) \quad k\geq0,
$$
with multiplicities $m_k=(k+1)^2$  \cite[Proposition 22.2 and Corollary 22.1]{shubin1987pseudodifferential}. Together with Lemma \ref{eigenvalue_Dirac_square}, this gives:

\begin{corollary}\label{multiplicity_eigenvalue}
$D^2$ has eigenvalues $(k+1)^2$ with multiplicities $4(k+1)^2$,  $k\geq0$.
\end{corollary}

In general, for an operator $T$ and a vector $u$ such that $T^2u=\mu^2u$,
if
$$
v^\pm:=(T\pm\mu) u \neq0
$$
then $v^\pm$ is an eigenvector of $T$ with eigenvalue $\pm\mu$. Let us apply this principle to $T=D$, with $\mu_k^2=(k+1)^2$ and $u_k=f_k\sigma_j$, for $j=1,\dots,4$. 

Let us first look at the case $k=0$, in which $f_0=1$ and $\lambda_0=0$, so $u_0=\sigma_j$ and $\mu_0^2=1$, i.e.,
$$
v^\pm=(D\pm\mu_0)\sigma_j=D\sigma_j\pm\sigma_j.
$$
Now, $\dirac\sigma_j=0$ by Lemma \ref{trivialized_normal_bundle}, so $D\sigma_j=-\sigma_j$ and therefore $v^+=0$ and $v^-=-2\sigma_j$. Accordingly, $v^-$ is an eigenvector of $D$ with eigenvalue $-\mu_0=-1$. Since $v^-=-2\sigma_j$, for $j=1,\dots,4$, the multiplicity of $-\mu_0=-1$ is at least $4$, but the multiplicity of $(-\mu_0)^2=\mu_0^2=1$ is already $4$, by Corollary \ref{multiplicity_eigenvalue}, therefore the multiplicity of $-\mu_0=-1$ is exactly $4$.

Now, for $k\geq 1$, we take $u_k=f_k\sigma_j$ and $\mu_k=k+1$, and
use the trivial fact that $e_i\times \sigma_j$ and $\sigma_j$ are linearly independent for all $i,j$:\begin{align*}
v^\pm_k=&(D\pm\mu_k)u_k=\dirac u_k-(1\mp\mu_k)u_k\\
       =&\sum_{i=1}^3e_i(f_k)e_i\times \sigma_j-\underbrace{(1\mp\mu_k)}_{\neq 0}\underbrace{f_k}_{\neq 0}\sigma_j\neq 0.
\end{align*}
Thus $v^\pm_k$ is an eigenvector of $D$ with eigenvalue $\pm\mu_k$, and it follows that $v^\pm$ is an eigenvector of $\dirac$ with eigenvalue $1\pm\mu_k$, such that $m(1+\mu_k)+m(1-\mu_k)=4(k+1)^2$.
 It remains to determine the multiplicities of the eigenvalues $1\pm(k+1)$.
We introduce the following notation:
$$
\mu^+_0:=1-\mu_0=0, \quad 
\mu^+_k:=1+\mu_k=k+2, \qandq 
\mu^+_{-k}:=1-\mu_k=-k, \quad k\geq 1.
$$
From Corollary \ref{multiplicity_eigenvalue}, multiplicities of opposite index add up as $m(\mu^+_k)+m(\mu^+_{-k})=4(k+1)^2$. Alternatively, in the sign convention of Remark \ref{sign-convention2}, we denote the eigenvalues of $\dirac$ by 
$$
\mu^-_0=0, \quad 
\mu^-_{-k}=-k-2, \qandq 
\mu^-_k=k, \quad k\geq 1,
$$
and again we know $m(\mu^-_k)+m(\mu^-_{-k})=4(k+1)^2$.
The multiplicities in both sign conventions satisfy the following relations:

\begin{lemma}\label{multiplicity_formula}
$$m(\mu^+_{-k})=m(\mu^-_k)=2(k+1)(k+2), \quad  k\geq 0.$$
and
$$ m(\mu^+_k)=m(\mu^-_{-k})=2k(k+1), \quad k\geq 1.$$ 
\end{lemma}

\begin{proof}
From the above, the operator $\dirac-\frac{3}{2}$ has eigenvalues 
$$
\alpha^+_0=-\frac{3}{2}, \quad 
\alpha^+_k=k+\frac{3}{2}-1 \qandq 
\alpha^+_{-k}=-k-\frac{3}{2}.
$$
Let $\alpha^-_k:=-\alpha^+_{-k}$. Since $\mu^-_k=-\mu^+_{-k}$, we have
$m(\alpha^\pm_k)=m(\mu^\pm_k)$, for all $k\in\mathbb{Z}$, and so 
$$
m(\alpha^\pm_k)+m(\alpha^\pm_{-k})=4(k+1)^2.
$$ 
Now the claim clearly holds  for $k=0$ and, by induction on $k \geq 1$, we have
\begin{align*}
m(\mu^+_{-(k+1)})&= m(\alpha^+_{-(k+1)})=4(k+2)^2-m(\alpha^+_{(k+1)})\\
                     &=4(k+2)^2-m(\alpha^-_k)=4(k^2+4k+4)-2(k+1)(k+2)\\
                     &=2(k+2)(k+3).
\end{align*}
To obtain the second equality we used the relation
$$
\alpha^+_{(k+1)}=(k+1)+\frac{3}{2}-1=\alpha^-_k,
$$
and for the last one we used the induction hypothesis on $\alpha^-_k$.
\end{proof}

The group $\Aut(S^7,\varphi)=\spin(7)$ of automorphisms of $S^7$ which fix the $\gt$--structure induces trivial associative deformations, and the associative $3$--sphere is invariant by the action of the embedded subgroup $K=\SU(2)\times\SU(2)\times\SU(2) /\mathbb{Z}_2\subset\spin(7)$, where $\mathbb{Z}_2$ is generated by $(-1,-1,-1)$ \cite[Theorem IV 1.38]{mclean1996deformations}. Therefore the space of infinitesimal associative deformations of $S^3$ has dimension at least $\dim(\spin(7)/K)=12$. 
\begin{corollary}
The $3$-sphere in $S^7$ is rigid as an associative submanifold.
\end{corollary}

\begin{proof}
Since $\mu^+_{-1}$ is the eigenvalue corresponding to the space of infinitesimal associative deformations, then, by Lemma \ref{multiplicity_formula}, $\dim(\ker \fueter)=m(\mu^+_{-1})=12$.
\end{proof}

\subsection{The example of Bryant and  Salamon}

In \cite{bryant1989construction}, Bryant and Salamon constructed an example of a $7$--manifold with constant scalar curvature and  holonomy exactly $\gt$:
\begin{theorem}
Let $(M^3,ds^2)$ be a Riemannian $3$-manifold with constant sectional curvature $K= 1$. Let $\mathbf{S}(M)\rightarrow M$ denote the standard spinor bundle, let $r:\mathbf{S}(M)\rightarrow \mathbb{R}$ be the squared Euclidean norm, and let $d\sigma^2$ denote the quadratic form of rank $4$ on the total space of $\mathbf{S}(M)\cong M^3\times \mathbb{R}^4$ which restricts to the standard flat metric on each  fibre and whose null space at each point is the horizontal space of the standard spin connection. Then the following metric on       $\mathbf{S}(M)$ 
is complete and has holonomy $\gt$:
\begin{equation}\label{metricbs}
g=3(r+1)^{2/3}ds^2+4(r+1)^{-1/3}d\sigma^2.
\end{equation}
Here $r=|a|^2=a\overline{a}$ denotes squared radial distance and $a: P_{\spin(3)}(M)\times \mathbb{H}\rightarrow \mathbb{H}$ is the projection onto the second factor.
\end{theorem}
The corresponding associative submanifold is of the form $S^3\times \{0\}$ for  $0\in \mathbb{R}^4$.
Now, a compact spin manifold of positive scalar curvature admits no harmonic spinors (see e.g. \cite{lawson1989spin}) -- in fact, the same conclusion holds if the scalar curvature is just nonnegative and somewhere positive
-- but McLean showed in \cite{mclean1996deformations} that the moduli space of associative deformations at $Y$ is the space of harmonic twisted spinors on $Y$, that is, the kernel of its Dirac operator. \

Observe that the normal bundle of $S^3\times \{ 0\}$ is isomorphic to the spinor bundle of $S^3$. In general, let $Y^3$ be an oriented Riemannian manifold and $\pi: P_{\SO}(Y)\rightarrow Y$ the frame bundle of oriented isometries. Now let $\xi: P_{\spin}(Y)\rightarrow P_{\SO}(Y)$ be the $\spin$ double cover of the bundle $P_{\SO}(Y)$. The normal bundle of $S^3$ (as a submanifold of $S^3\times \mathbb{R}^4$) can be written as an associated bundle $NS^3=P_{SO}(S^3)\times_{\SO(4)}\mathbb{H}$, via the representation $\varrho: \SO(4)\rightarrow \Gl(\mathbb{H})$, $\varrho([p,q])(v)=pv\bar{q}$. Now, the spinor bundle of $S^3$ can be written as $\mathbf{S}(S^3)=P_\spin(S^3)\times_{\spin(3)}\mathbb{H}$ via the inclusion $\iota^-: p\in \spin(3)\hookrightarrow (1,p)\in \spin(4)$ and the representation $\varsigma: \spin(3)\times \spin(3)\rightarrow \Gl(\mathbb{H})$, $\varsigma(p,q)(v)=v\bar{q}$. So the identification $\mathbb{R}^4\cong \mathbb{H}$ gives a bundle map 
$$
\Phi: \mathbf{S}(S^3) \rightarrow NS^3
$$
by $\Phi(\widetilde{p},v)=(\xi(\widetilde{p}),v)$. Observe that $\Phi$ is well-defined:
\begin{align*}
\Phi(pg^{-1},\varsigma(1,g)(v))&=(\xi(p\cdot g),v\bar{g})\\
                                        &=(\xi(p)\xi_0(g),\varrho([1,g])(v))\\
                                        &=(\xi(p),v)\cdot g=\Phi(p,v).
\end{align*}
It is easy to check that $\Phi$ is  a bundle isomorphism.

In those terms, we obtain a trivial alternative for Gayet's proof of rigidity
of
 $S^3$ \cite{gayet2014smooth}, using Theorem \ref{rigidity_theorem} and the fact that $\tau_0=0$.

\begin{prop}
Let $S^3\times \mathbb{R}^4$ be the $\gt$--manifold with the metric \eqref{metricbs}.  Then $S^3\cong S^3\times \{0\}$ is rigid as an associative submanifold. \begin{proof}
Since the normal bundle of $S^3\times \{0\}$ coincides with the Spin bundle $\mathbf{S}(S^3)$, the term $\overline{\rho}(F^-)$ vanishes, and we conclude
immediately from Theorem \ref{rigidity_theorem}.
\end{proof}
\end{prop}

\subsection{Locally conformal calibrated case and applications}

As an application of the Fueter-Dirac Weitzenb\"ock formula (\ref{fueter_dirac_formula}) and Proposition \ref{alternative_expression}, we focus on \emph{locally conformal calibrated} $\gt$--structures, whose associated metric is (at least locally) conformal  to a metric induced by a calibrated  $\gt$--structure. We provide a novel  example of a rigid associative submanifold, inside a compact manifold $S$ with a locally conformal calibrated $\gt$--structure, studied by Fernández, Fino and Raffero \cite{fernandez2016locally}. 

\begin{definition}
A $\gt$--structure is  \emph{locally conformal calibrated} if it has vanishing torsion components $\tau_0\equiv 0$ and $\tau_3\equiv 0$, so 
\begin{align*}
  d\varphi &=3\tau_1\wedge\varphi,\\
  d\psi    &=4\tau_1\wedge\psi+\tau_2\wedge\varphi.
\end{align*}
\end{definition}

A $\SU(3)$--\emph{structure} on a $6$-manifold $N$ is a pair  $(\omega,\phi_+)\in \Omega^2(N)\times\Omega^3(N)$ such that $\phi_+=\frac{1}{2}(\Omega+\bar{\Omega})$, where $\Omega\in \Omega^0(\Lambda^3(T^\ast N\otimes \mathbb{C}))$ is a decomposable complex $3$-form and 
$$
\omega\wedge\phi_+=0 \qandq \frac{\omega^3}{6}=\frac{i}{8}\Omega\wedge\bar{\Omega}=\frac{1}{4}\phi_+\wedge\phi_- \qwithq \phi_-:=\frac{1}{2i}(\Omega-\bar{\Omega}). 
$$
The $\SU(3)$--structure $(\omega,\phi_+)$ is said to be \emph{coupled} if $d\omega=c\phi_+$ with $c$ a non-zero real number. So, the product manifold $N\times S^1$ has a natural locally conformal calibrated $\gt$--structure defined by
$$
  \varphi=\omega\wedge dt +\phi_+,
$$
with $\tau_0\equiv 0, \tau_3\equiv 0$ and $\tau_1=-\frac{c}{3}dt$.  

\begin{example}\cite[Example 3.3]{fernandez2016locally}
\label{solv_example} 
Consider the $6$--dimensional Lie algebra $\mathfrak{n}_{28}$, and let $\{e_1,...,e_6\}$ be a $\SU(3)$--basis. With respect to the dual basis $\{e^1,...,e^6\}$, the structure equations of $\mathfrak{n}_{28}$ are 
\begin{equation}\label{Heisenberg_structure_eq}
  (0,0,0,0,e^{13}-e^{24},e^{14}+e^{23}),
\end{equation}
and we denote its components by $de^i:=0$, for $i=1,\dots,4$, $de^5:=e^{13}-e^{24}$ and $de^6:=e^{14}+e^{23}$. The pair 
\begin{equation}\label{h_su3-structure}
  \omega=e^{12}+e^{34}-e^{56} \qandq \phi_+=e^{136}-e^{145}-e^{235}-e^{246}
\end{equation}
defines a coupled $SU(3)$--structure on $\mathfrak{n}_{28}$ with $d\omega=-\phi_+$. Denote by $G$ the $3$-dimensional complex Heisenberg group with Lie algebra $\Lie(G)=\mathfrak{n}_{28}$ given by
$$
G=\Bigg\{\begin{pmatrix}
1 & z_1 & z_3 \\
0 &  1  & z_2 \\
0 &  0  &  1
\end{pmatrix}; \ \ z_1,z_2,z_3\in \mathbb{C} \Bigg\}.
$$
The structure equations \eqref{Heisenberg_structure_eq} can be rewritten as
$$
  dz_1=e^1+ie^2, \ \ \ dz_2=e^3+ie^4 \ \ \ dz_3+z_1dz_2=e^5+ie^6.
$$
By \cite[Theorem 7]{mal1949class}, $G$ admits a uniform discrete subgroup   $\Gamma\subset G$, i.e., a discrete subgroup such that $\Gamma\backslash G$ is compact,  the elements of which have $z_1,z_2,z_3\in\mathbb{Z}[i]$. The left-invariant forms $\omega$ and $\phi_+$ on $G$ are well defined in the quotient $\Gamma\backslash G$. Consider the automorphism $\nu: G\rightarrow G$  defined by
$$
\begin{pmatrix}
1 & z_1 & z_3 \\
0 &  1  & z_2 \\
0 &  0  &  1
\end{pmatrix} \xrightarrow[{  }]{\nu} \begin{pmatrix}
1 & iz_1 & z_3 \\
0 &  1  & -iz_2 \\
0 &  0  &  1
\end{pmatrix},
$$
and denote by $\Diff_\nu:=\left\langle (p,t)\mapsto (\nu(p),t+1)\right\rangle$ the infinite cyclic subgroup of diffeomorphisms of $(\Gamma\backslash G)\times \mathbb{R}$. The manifold 
$$
S=\Big( (\Gamma\backslash G)\times \mathbb{R}\Big)/\Diff_\nu
$$ 
is endowed with a locally conformal calibrated $\gt$--structure as follows: for the left-invariant coframe given in \eqref{Heisenberg_structure_eq}, we have
$$
\nu^\ast(e_1)=-e_2, \ \nu^\ast(e_2)=e_1, \ \nu^\ast(e_3)=e_4, \ \nu^\ast(e_4)=-e_3, \ \nu^\ast(e_5)=e_5, \ \nu^\ast(e_6)=e_6.
$$
Hence $\nu^\ast\omega=\omega$ and $\nu^\ast\phi_+=\phi_+$, for $(\omega,\phi_+)$ defined in \eqref{h_su3-structure}. Denoting by  $p_1: (\Gamma\backslash G)\times\mathbb{R}\rightarrow \Gamma\backslash G$  the projection onto the first factor, the forms $p_1^\ast \omega\in \Omega^2((\Gamma\backslash G)\times \mathbb{R})$ and $p_1^\ast \phi_+\in\Omega^3((\Gamma\backslash G)\times \mathbb{R})$ are invariant under $\sim_\nu$. Therefore, we have differential forms $\widetilde{\omega}\in\Omega^2(S)$ and $\widetilde{\phi}_+\in\Omega^3(S)$ satisfying the same relations as $(\omega,\phi_+)$ from \eqref{h_su3-structure}.
In this setup, the $3$-form
\begin{equation}\label{lcc_g2_structure}
  \widetilde{\varphi}=\widetilde{\omega}\wedge e^7+\widetilde{\phi}_+
\end{equation}
defines a locally conformal calibrated $\gt$-structure on $S$. Here $e^7$ denotes the pullback of the canonical closed $1$-form on $\mathbb{R}$ by the projection $p_2:(\Gamma\backslash G)\times\mathbb{R}\rightarrow \mathbb{R}$.
The torsion forms of $\widetilde{\varphi}$ are
$$
  \tau_1=\frac{1}{3}e^7, \quad \tau_2=\widetilde{\alpha} \qwhereq \alpha=-\frac{4}{3}\bigg(e^{12}+e^{34}+ 2e^{56}\bigg)
$$
and, by Proposition \ref{prop: full torsion tensor}, the full torsion tensor
is$$
T=\widetilde{\beta}, \qwithq \beta=e^{12}+e^{34}+e^{56}.
$$

\end{example}

The $7$-manifold from Example \ref{solv_example} contains an associative submanifold, corresponding to  a particular Lie subalgebra:

\begin{example}\label{lie_associative_subgroup}
Consider the Abelian subalgebra $\mathfrak{n}_{28}^'=\Span(e_5,e_6)\subset \mathfrak{n}_{28}$ and its respective Lie group $G^'=[G,G]= \exp (\mathfrak{n}_{28}^')\subset G$,  which is generated by the commutator $[g,h]=ghg^{-1}h^{-1}$. Since $G'$ is obtained as the maximal integral submanifold of $G$ given by the left-invariant distribution
$$
  \Delta(g)=(dL_g)_1\mathfrak{n}_{28} \qforq g\in G,
$$
i.e. $(L_h)_\ast(\Delta(g))\subset \Delta(hg)$ (c.f. \cite[Theorem 6.5]{sanmartingruposdeLie}), we get an integral distribution $\bar{\Delta}$ on $\Gamma\backslash G$. Representing $G'$ by 
$$
G'=\Bigg\{\begin{pmatrix}
1 &  0  & z_3 \\
0 &  1  &  0  \\
0 &  0  &  1
\end{pmatrix}; \ \ z_3\in \mathbb{C} \Bigg\},
$$
we see that, for each $p=\Gamma g'\in \Gamma\backslash G'$, we have $T_p(\Gamma\backslash G')=\bar{\Delta}(\Gamma g')$, and so $\Gamma\backslash G'$ is a compact embedded submanifold of $\Gamma\backslash G$.
Now $\nu|_{G'}=Id$ and the quotient map $(\Gamma\backslash G)\times \mathbb{R}\rightarrow S$ is a local diffeomorphism, so  
$$
  Y=\bigg( (\Gamma\backslash G')\times \mathbb{R}\bigg)/\Diff_\nu \cong (\Gamma\backslash G')\times S^1
$$
is a compact embedded submanifold of $S$.
Moreover,
$$
  T_{(p,t)}Y=T_p(\Gamma\backslash G')\oplus T_t \mathbb{R} \cong \mathfrak{n}_{28}'\oplus \mathbb{R},
$$ 
and indeed $\widetilde{\varphi}|_{T_pY}\equiv \vol(e_5,e_6,e_7)$. Hence, $Y$ is a closed associative submanifold of $S$.
\end{example}

Now, we  assess formula \eqref{fueter_dirac_operator} of Section \ref{1st_corrections} for  Example \ref{lie_associative_subgroup}. The first correction term is
\begin{align*}
P_1(\sigma)&=-T_{56}e_5\times\nabla_6^\perp\sigma-T_{65}e_6\times\nabla_5^\perp\sigma-2T_{56}\nabla_7^\perp\sigma\\
           &=-(e_7\times e_6)\times\nabla_6^\perp\sigma-(e_7\times e_5)\times\nabla_5^\perp\sigma-2\nabla_7^\perp\sigma\\
           &= e_7\times \dirac (\sigma)-\nabla_7^\perp\sigma.
\end{align*}
Here, to obtain the second equality we used the associative relation $e_5\times e_6=-e_7$ and for the last one we used the identity $(u\times v)\times w=-u\times(v\times w)$, for  mutually orthonormal $u,v,w$.
To calculate $P_2$, we need the covariant derivative of the total torsion tensor $T$
\begin{equation}\label{derivative_torsion_tensor}
  \nabla_i T_{kl}=e_i(T_{kl})-\Gamma_{ik}^mT_{ml}-\Gamma_{il}^mT_{km}=-\Gamma_{ik}^mT_{ml}-\Gamma_{il}^mT_{km}.
\end{equation}
Since $S$ is locally isometric to $G\times \mathbb{R}$,  the Christoffel symbols of the   $\gt$-metric on $S$ are defined by the structure constants of the Lie algebra $\mathfrak{n}_{28}$ (cf. \cite{milnor1976curvatures}):
$$
  \Gamma_{ij}^k=\frac{1}{2}(\alpha_{ijk}-\alpha_{jki}+\alpha_{kij}) \qwithq \alpha_{ijk}=\langle [e_i,e_j],e_k\rangle.
$$
Applying this to Example \ref{solv_example}, we find
\begin{align*}
        \Gamma_{13}^5 = \Gamma_{23}^6 = \Gamma_{36}^2 = \Gamma_{42}^5 =
        \Gamma_{63}^2 = \Gamma_{52}^4 &= 
        -\frac{1}{2}\\
        \Gamma_{14}^6 = \Gamma_{25}^4 = \Gamma_{35}^1 = \Gamma_{46}^1 =
        \Gamma_{64}^1 = \Gamma_{53}^1 
        &=-\frac{1}{2}       
        \\
        \Gamma_{16}^4 = \Gamma_{24}^5 = \Gamma_{31}^5 = \Gamma_{41}^6 =
        \Gamma_{61}^4 = \Gamma_{51}^3  
        &= +\frac{1}{2}\\
        \Gamma_{15}^3 = \Gamma_{26}^3 = \Gamma_{32}^6 = \Gamma_{45}^2 =
        \Gamma_{62}^3 = \Gamma_{54}^2 
        &=+\frac{1}{2}
        \\
        \\\Gamma_{ij}^k&=0, \; \text{otherwise}.
\end{align*}
Using the cross product defined by \eqref{lcc_g2_structure} and the above Christoffel symbols, we have:
\begin{equation}\label{connection_endomorphism}
\nabla_le_{i+5}= \nabla_{i+5} e_l=\frac{(-1)^i}{2}e_{6-i}\times e_l \qforq i=0,1 \qandq l=1,2,3,4.
\end{equation}
Notice that the full torsion tensor of the $\gt$--structure \eqref{lcc_g2_structure} can be written as
\begin{equation}\label{Torsion_endomophism}
T(u,v)=-\langle e_7\times u^\top,v^\top\rangle+\langle e_7\times u^\perp,v^\perp\rangle  \qforq u,v\in \Omega^0(TS|_Y)=\Omega^0(TY)\oplus \Omega^0(NY),
\end{equation}
where $u^\top$ and $u^\perp$ are the tangent and  normal components of $u$, respectively. Combining these facts with Lemma \ref{Leibniz_rule} $(i)$, we have
\begin{align}\label{Leibniz_rule-lcc}
\begin{split}
\nabla_u(v\times w)&=\nabla_uv\times w+v\times\nabla_uw+\sum_{i=1}^7T(u,e_m)\chi(e_m,v,w)\\
                   &=\nabla_uv\times w+v\times\nabla_uw-\chi(e_7\times u^\top,v,w)+\chi(e_7\times u^\perp,v,w).
\end{split}
\end{align}
Now, for $P_2$ we obtain:
\begin{align*}
P_2(\sigma)&=\sum_{i=5}^7\sum_{k=1}^4e_i(T(\sigma,e_k))e_i\times e_k\\
           &=\sum_{i=5}^7\sum_{k=1}^4e_i\times\Big( \nabla_i^\perp(T(\sigma,e_k)e_k)-T(\sigma,e_k)\nabla_i^\perp e_k\Big)\\
           &=\sum_{i=5}^7e_i\times \nabla_i^\perp(e_7\times\sigma)-\sum_{i=0,1}\sum_{k=1}^4\langle e_7\times\sigma,e_k\rangle\frac{(-1)^i}{2}e_{i+5}\times(e_{6-i}\times e_k)\\
           &=\sum_{i=5}^7e_i\times (e_7\times\nabla_i^\perp\sigma)-e_i\times\chi(e_7\times e_i,e_7,\sigma)-\sum_{i=0,1}\frac{(-1)^i}{2}e_{i+5}\times(e_{6-i}\times(e_7\times\sigma) )\\
           &=-2\nabla_7^\perp\sigma+\sum_{i=5}^7-e_7\times (e_i\times\nabla_i^\perp\sigma)\\
            &\quad-\underbrace{\sum_{i=0,1}e_{i+5}\times\chi(e_7\times e_{i+5},e_7,\sigma)+\frac{(-1)^i}{2}(e_{i+5}\times e_{6-i})\times(e_7\times\sigma)}_{(\star)}\\
           &=-e_7\times\dirac(\sigma)-2\nabla_7^\perp\sigma-3\sigma
\end{align*}
For the third equality, we used \eqref{Torsion_endomophism} in the first term and \eqref{connection_endomorphism}  in the second one. The fourth equality follows from \eqref{Leibniz_rule-lcc} and, finally, a short calculation gives:
\begin{align*}
        (\star)&=\sum_{i=0,1}-e_{i+5}\times((e_7\times e_{i+5})\times(e_7\times\sigma))+\frac{(-1)^i}{2}(e_{i+5}\times e_{6-i})\times(e_7\times\sigma)\\
               &=\sum_{i=0,1}-((e_{i+5}\times e_7)\times e_{i+5})\times(e_7\times\sigma)+\frac{(-1)^i}{2}(e_{i+5}\times e_{6-i})\times(e_7\times\sigma)\\
               &=-((e_5\times e_7)\times e_5)\times(e_7\times\sigma)+\frac{1}{2}(e_5\times e_6)\times(e_7\times\sigma)\\
                & \quad-((e_6\times e_7)\times e_6)\times(e_7\times\sigma)-\frac{1}{2}(e_6\times e_5)\times(e_7\times\sigma)\\
               &= \sigma+\frac{1}{2}\sigma +\sigma+\frac{1}{2}\sigma=3\sigma. 
\end{align*}
Finally, for $P_3$, we have
\begin{align*}
  P_3(\sigma)&=\sum_{k,l=1}^4\Big(T(\sigma,e_k)+\sum_{i=5}^7\widetilde{\varphi}(e_i,\nabla_i^\perp\sigma,e_k)\Big)T_{kl}e_l\\
             &=\sum_{k=1}^4\Big(\langle e_7\times \sigma,e_k\rangle+\sum_{i=5}^7\langle e_i\times\nabla_i^\perp\sigma,e_k\rangle\Big)e_7\times e_k\\
             &= e_7\times(e_7\times\sigma)+e_7\times\dirac(\sigma)=-\sigma+e_7\times\dirac(\sigma)
\end{align*}

Now, writing the curvature tensor as
$$
R(e_i,e_j)e_k 
= 
\sum_{l,m=1}^7\Big(\Gamma_{jk}^l\Gamma_{il}^m-\Gamma_{ik}^l\Gamma_{jl}^m-(\Gamma_{ij}^l-\Gamma_{ji}^l)\Gamma_{lk}^m\Bigr)e_m
$$
and using the last expression, we have
\begin{align*}
  R(e_5,\sigma)e_5&=\sum_{l,m=1}^7\sum_{j=1}^4\sigma^j\Big(\Gamma_{j5}^l\Gamma_{5l}^m-\Gamma_{55}^l\Gamma_{jl}^m-(\Gamma_{5j}^l-\Gamma_{j5}^l)\Gamma_{l5}^m\Bigr)e_m\\
                  &=\sum_{l,m=1}^7\sum_{j=1}^4\sigma^j\Big(\Gamma_{j5}^l\Gamma_{5l}^m\Bigr)e_m\\
                  &=\sigma^1\Gamma_{15}^3\Gamma_{53}^1e_1+\sigma^2\Gamma_{25}^4\Gamma_{54}^2e_2+\sigma^3\Gamma_{35}^1\Gamma_{51}^3e_3+\sigma^4\Gamma_{45}^2\Gamma_{52}^4e_4=-\frac{\sigma}{4}.
\end{align*}
And,
\begin{align*}
R(e_6,\sigma)e_6&=\sum_{l,m=1}^7\sum_{j=1}^4\sigma^j\Big(\Gamma_{j6}^l\Gamma_{6l}^m-\Gamma_{66}^l\Gamma_{jl}^m-(\Gamma_{6j}^l-\Gamma_{j6}^l)\Gamma_{l6}^m\Bigr)e_m\\
        &=\sum_{l,m=1}^7\sum_{j=1}^4\sigma^j\Big(\Gamma_{j6}^l\Gamma_{6l}^m\Bigr)e_m\\
        &=\sigma^1\Gamma_{16}^4\Gamma_{64}^1e_1+\sigma^2\Gamma_{26}^3\Gamma_{63}^2e_2+\sigma^3\Gamma_{36}^2\Gamma_{62}^3e_3+\sigma^4\Gamma_{46}^1\Gamma_{61}^4e_4=-\frac{\sigma}{4}.
\end{align*}
Therefore,
\begin{align*}
\mathcal{R}(\sigma)&=\Bigl( R(e_5,\sigma)e_5+R(e_6,\sigma)e_6+R(e_7,\sigma)e_7\Bigr)^\perp
                   =-\frac{1}{4}\sigma-\frac{1}{4}\sigma+0\\
                   &=-\frac{1}{2}\sigma. 
\end{align*}
Now, we assess the operator $\mathcal{T}$ defined in equation \eqref{torsion_four_tensor} for a pair $e_i, e_j\in \Omega^0(TY)$ and $\sigma\in \Omega^0(NY)$:
\begin{align*}
\mathcal{T}(e_j,\sigma,e_i,e_j)=&\sum_{m=1}^7\underbrace{T(\sigma,e_m)\nabla_j\psi(e_m,e_i,e_j,\cdot)^\sharp}_{\mathrm{(I)}}-\underbrace{T(e_j,e_m)\nabla_\sigma\psi(e_m,e_i,e_j,\cdot)^\sharp}_{\mathrm{(II)}}\\
                         &+\underbrace{\Big(\nabla_jT(\sigma,e_m)-\nabla_\sigma T(e_j,e_m)\Big)\chi(e_m,e_i,e_j)}_{\mathrm{(III)}}.
\end{align*}
We will use throughout the proof both the expression of $\nabla\psi$ in terms of $T$ and $\varphi$ from Corollary \ref{covariant_psi} and the expression for $T$ given in \eqref{Torsion_endomophism}. For  the first term, 
\begin{align*}
\mathrm{(I)}
        &=\sum_{m=1}^7\langle e_7\times \sigma, e_m\rangle\nabla_j\psi(e_m,e_i,e_j,\cdot)^\sharp=\nabla_j\psi(e_7\times\sigma,e_i,e_j,\cdot)^\sharp\\
        &=-T(e_j,e_7\times\sigma)\varphi(e_i,e_j,\cdot)^\sharp+T(e_j,e_i)\varphi(e_7\times\sigma,e_j,\cdot)^\sharp-T(e_j,e_j)\varphi(e_7\times\sigma,e_i,\cdot)^\sharp\\
        &\quad+T(e_j,\cdot)^\sharp\varphi(e_7\times\sigma,e_i,e_j)\\
        &=-\langle e_7\times e_j,e_i\rangle(e_7\times \sigma)\times e_j=\langle e_7\times e_j,e_i\rangle(e_7\times e_j)\times \sigma.
\end{align*}
Here we used the vanishings $T(e_j,e_7\times\sigma)=0$, again by  \eqref{Torsion_endomophism}, $T(e_j,e_j)=0$, by skew-symmetry, and $\varphi(e_7\times\sigma,e_i,e_j)=\langle e_i\times e_j, e_7\times\sigma\rangle=0$, by orthogonality. 

For the second term,
\begin{align*}
\mathrm{(II)}=& \sum_{m=1}^7\langle e_7\times e_j, e_m\rangle\nabla_\sigma\psi(e_m,e_i,e_j,\cdot)^\sharp=\nabla_\sigma\psi(e_7\times e_j,e_i,e_j,\cdot)^\sharp\\
   =& -T(\sigma,e_7\times e_j)\varphi(e_i,e_j,\cdot)^\sharp+T(\sigma,e_i)\varphi(e_7\times e_j,e_j,\cdot)^\sharp-T(\sigma,e_j)\varphi(e_7\times e_j,e_i,\cdot)^\sharp\\
    & +T(\sigma,\cdot)^\sharp\varphi(e_7\times e_j,e_i,e_j)\\
   =& -\langle e_7\times\sigma,\cdot\rangle^\sharp\langle(e_7\times e_j)\times e_i,e_j\rangle=-\langle(e_7\times e_j)\times e_i,e_j\rangle e_7\times \sigma.
\end{align*}
Again the vanishings $T(\sigma,e_7\times e_j)=T(\sigma,e_i)=T(\sigma,e_j)=0$ follow from \eqref{Torsion_endomophism}. 

For the third term, we use expression  \eqref{derivative_torsion_tensor} for the derivatives of the torsion tensor:
\begin{align*}
\mathrm{(III)}=&-\sum_{m=1}^7\Big(T(\sigma,\nabla_je_m)-T(e_j, \nabla_\sigma e_m)\Big)\chi(e_m,e_i,e_j)\\
     =&-\sum_{m=1}^7\Big(\langle e_7\times\sigma,\nabla_je_m\rangle+\langle e_7\times e_j, \nabla_\sigma e_m\rangle\Big)\chi(e_m,e_i,e_j).
\end{align*}
We now apply $\mathrm{(I),(II)}$ and $(\mathrm{III)}$ for $i=5$ and $j=6$:
\begin{align*}
\mathcal{T}(e_6,\sigma,e_5,e_6)
        &= \langle e_7\times e_6,e_5\rangle(e_7\times e_6)\times\sigma+\langle (e_7\times e_6)\times e_5, e_6\rangle e_7\times\sigma\\
        &\quad -\sum_{m=1}^7\Big(\langle e_7\times\sigma,\nabla_6e_m\rangle+\langle e_7\times e_6,\nabla_\sigma e_m\rangle\Big)\chi(e_m,e_5,e_6)\\
        &= e_5\times\sigma-\sum_{m=1}^7\Big(-\frac{1}{2}\langle e_7\times\sigma,e_5\times e_m\rangle+\langle e_5,\nabla_\sigma e_m\rangle\Big)\chi(e_m,e_5,e_6)\\ 
        &= e_5\times\sigma-\sum_{m=1}^7\Big(\frac{1}{2}\langle e_5\times (e_7\times\sigma), e_m\rangle+\sigma\langle e_5,e_m\rangle-\langle\nabla_\sigma e_5,e_m\rangle\Big)\chi(e_m,e_5,e_6)\\
        &= e_5\times\sigma-\sum_{m=1}^7\Big(-\frac{1}{2}\langle e_6\times\sigma, e_m\rangle-\frac{1}{2}\langle e_6\times\sigma ,e_m\rangle\Big)\chi(e_m,e_5,e_6)\\
        &= e_5\times\sigma+\chi(e_6\times\sigma,e_5,e_6)=e_5\times\sigma-(e_6\times\sigma)\times(e_5\times e_6)\\
        &= e_5\times\sigma+(e_6\times\sigma)\times e_7\\
        &=2e_5\times \sigma.
\end{align*}
Here we used repeatedly that $e_5\times e_6=-e_7$ and $e_i\times(e_j\times\sigma)=-e_j\times(e_i\times\sigma)$ for $i\neq j$. At the second and fourth lines  we applied again \eqref{connection_endomorphism}, and at the third line we used the compatibility of the Riemannian connection.
 
For $j=7$ and $i=6$, we have trivially
\begin{align*}
\mathcal{T}(e_7,\sigma,e_6,e_7)=0.
\end{align*}
Finally, for $j=5$ and $i=7$, we have
\begin{align*}
\mathcal{T}(e_5,\sigma,e_7,e_5)
        &= \langle e_7\times e_5,e_7\rangle(e_7\times e_5)\times\sigma+\langle (e_7\times e_5)\times e_7, e_5\rangle e_7\times\sigma\\
        &\quad -\sum_{m=1}^7\Big(\langle e_7\times\sigma,\nabla_5e_m\rangle+\langle e_7\times e_5,\nabla_\sigma e_m\rangle\Big)\chi(e_m,e_7,e_5)\\
        &= \langle e_6, e_7\rangle e_6\times\sigma-\langle e_6\times e_7,e_5\rangle e_7\times\sigma-\sum_{m=1}^7\Big(\frac{1}{2}\langle e_7\times\sigma,e_6\times e_m\rangle\\
        &\quad -\langle e_6,\nabla_\sigma e_m\rangle\Big)\chi(e_m,e_7,e_5)\\
        &= e_7\times\sigma-\sum_{m=1}^7\Big(-\frac{1}{2}\langle e_6\times( e_7\times\sigma), e_m\rangle-\sigma\langle e_6, e_m\rangle\\
        &\quad +\langle \nabla_\sigma e_6,e_m\rangle\Big)\chi(e_m,e_7,e_5)\\
        &= e_7\times\sigma-\sum_{m=1}^7\Big(-\frac{1}{2}\langle e_5\times\sigma, e_m\rangle-\frac{1}{2}\langle e_5\times\sigma ,e_m\rangle\Big)\chi(e_m,e_7,e_5)\\
        &= e_7\times\sigma+\chi(e_5\times\sigma,e_7,e_5)= e_7\times\sigma -(e_5\times\sigma)\times(e_7\times e_5)\\
        &= e_7\times\sigma +(e_5\times\sigma)\times e_6=2e_7\times \sigma.
\end{align*}
Therefore, 
\begin{eqnarray*}
\biggl(\sum_{i\in\mathbb{Z}_3}e_{i+5}\times\mathcal{T}(e_{i+6},\sigma,e_{i+5},e_{i+6})\biggr)^\perp
        &=&-4\sigma.
\end{eqnarray*}
Following the notation of \cite[\S$5.3$]{corti1207g2}, we define an operator 
$$ 
  \dirac^c(\sigma):=e_5\times\nabla_5^\perp \sigma+e_6\times\nabla_6^\perp\sigma,
$$
and recall that the cross-product by $e_7$ defines an almost complex structure on $T(\Gamma\backslash G)$ denoted by $J(\sigma):=e_7\times \sigma$. Then \eqref{D_with_torsion} becomes
$$
  \fueter(\sigma)=\dirac^c(\sigma)+J(\dot{\sigma})+J(\sigma),
$$
where $\dot{\sigma}:=\n_7^\perp\sigma$. To simplify notation, let $\left\Vert \cdot \right\Vert$ and $\Langle \cdot ,\cdot\Rangle$ denote the $L^2$-norm and inner product of sections, respectively (the integral of the corresponding pointwise quantity over the associative submanifold). The next Lemma gathers some relations between the operators $\dirac$, $J$ and $\n$; although some of them will not be used in this article, we state them anyway as a curiosity.
\begin{lemma}\label{support_lemma }
With the above notation, we have the following properties: 
\begin{enumerate}
        \item[$\mathrm{(i)}$]
 $\dirac^c \circ J(\sigma)=-J\circ\dirac^c(\sigma)+2\sigma$.

        \item[$\mathrm{(ii)}$]
        $\Langle\dirac^c(\sigma),\eta\Rangle=\Langle\sigma,\dirac^c(\eta)\Rangle+2\Langle\sigma,J(\eta)\Rangle$.

        \item[$\mathrm{(iii)}$]
        $ \Langle\dirac^c(\sigma),J(\dot{\sigma})\Rangle=0$.  
        \item[$\mathrm{(iv)}$]
        $\Langle \dot{\sigma}, \sigma\Rangle = 0$ and $\Langle \dirac^c(\sigma), J(\sigma)\Rangle \leq 0$.
\end{enumerate}
\end{lemma}

\begin{proof}
\begin{enumerate}[(i)]
        \item
        Using Lemma \ref{Leibniz_rule} $(i)$, we have, 
\begin{eqnarray*}
\dirac^c\circ J(\sigma)&=&-J\circ \dirac^c(\sigma)-T_{65}e_6\times(e_5\times(e_7\times \sigma))-T_{56}e_5\times(e_6\times(e_7\times\sigma))\\
                       &=&-J\circ \dirac^c(\sigma)+2T_{56}(e_5\times e_6)\times (e_7\times \sigma)\\
                       &=&-J\circ\dirac^c (\sigma)+2\cdot \sigma. 
\end{eqnarray*}

        \item
\begin{eqnarray*}
\langle \dirac^c(\sigma),\eta\rangle_p &=& -\sum_{i=5}^6\langle \nabla_i^\perp \sigma, e_i\times \eta\rangle_p= -\sum_{i=5}^6 \{e_i\langle \sigma,e_i\times \eta\rangle-\langle\sigma,\nabla_i^\perp(e_i\times\eta)\rangle\}_p\\
                                       &=& \dv(\sigma\times \eta)_p +  -\sum_{i=5}^6 \langle\sigma, e_i\times\nabla_i^\perp \eta -\chi(e_7\times e_i,e_i,\eta) \rangle \}_p\\   
                                       &=& \dv(\sigma\times \eta)_p+\langle \sigma,\dirac^c(\eta)\rangle_p+2\langle \sigma, e_7\times \eta\rangle_p.                      
\end{eqnarray*}
Here we used the Leibniz rule  \eqref{Leibniz_rule-lcc}, then the following trivial calculation:
\begin{align*}
\chi(e_7\times e_i, e_i,\eta)=&\chi(\eta,e_7\times e_i,e_i)=-\eta\times\Big( (e_7\times e_i)\times e_i\Big)\\
                             =&-\eta\times\Big(e_i\times(e_i\times e_7)\Big)=-e_7\times\eta.
\end{align*}
        \item
        Using  (i) and (ii), one has $\Langle\dirac^c(\sigma),J(\dot{\sigma})\Rangle=\Langle J(\sigma),\dirac^c(\dot{\sigma})\Rangle$, and, by the vanishing of the normal curvature tensor $R^\perp(e_i,e_7)\sigma=0$ for $i=5,6$, we have $\nabla_i^\perp\nabla_7^\perp\sigma=\nabla_7^\perp\nabla_i^\perp\sigma$. Using  Lemma \ref{Leibniz_rule} (i) and the compatibility of $\nabla^\perp$ with the induced metric in $NY$we have
\begin{eqnarray*}
\langle\dirac^c(\sigma),J(\dot{\sigma})\rangle_p&=& \sum_{i=5}^7\langle J(\sigma),e_i\times\nabla_7^\perp\nabla_i^\perp\sigma\rangle_p\\
                                        &=& \sum_{i=5}^7\langle J(\sigma),\nabla_7^\perp(e_i\times\nabla_i^\perp\sigma)\rangle_p\\
                                        &=& -\langle \nabla_7^\perp(J(\sigma)),\dirac^c(\sigma)\rangle_p +e_7\langle J(\sigma),\dirac^c(\sigma)\rangle_p\\
                                        &=&-\langle J(\dot{\sigma}),\dirac^c(\sigma)\rangle_p +\dv(\langle J(\sigma),\dirac^c(\sigma)\rangle e_7)_p.
\end{eqnarray*}

        \item
       Again by compatibility of $\nperp$ with the metric on $NY$, we have $2\langle \dot{\sigma},\sigma\rangle=2\langle\n_7^\perp\sigma,\sigma\rangle=e_7|\sigma|^2$. Now Stokes' Theorem gives
        \begin{equation}\label{div_vanishes}
         \Langle \dot{\sigma},\sigma\Rangle=\frac{1}{2}\int_Ye_7|\sigma|^2d\vol_Y=\frac{1}{2}\int_Y\dv(|\sigma|^2e_7)d\vol_Y=0. 
        \end{equation}
        Computing the $L^2$-norm for $\fueter(\sigma)$, we have 
\begin{align*}
      \left\Vert\fueter(\sigma)\right\Vert^2 
      =&\left\Vert\dirac^c(\sigma)\right\Vert^2
        +\left\Vert\dot{\sigma}\right\Vert^2
        +\left\Vert\sigma\right\Vert^2
        +2\Langle\dirac^c(\sigma),J(\dot{\sigma})\Rangle+2\Langle\dirac^c(\sigma),J(\sigma)\Rangle
        +2\Langle\dot{\sigma},\sigma\Rangle,
\end{align*}
and from Lemma \ref{support_lemma }(iii) and equation \eqref{div_vanishes} it follows that 
        \begin{align*}
        \left\Vert\fueter(\sigma)\right\Vert^2 
        =&\left\Vert\dirac^c(\sigma)\right\Vert^2
        +\left\Vert\dot{\sigma}\right\Vert^2
        +\left\Vert\sigma\right\Vert^2
        +2\Langle\dirac^c(\sigma),J(\sigma)\Rangle.
        \end{align*}
Therefore, by the triangle inequality,     
\begin{align*}
        \Langle\dirac^c(\sigma),J(\sigma)\Rangle\leq 0. 
\end{align*}
\end{enumerate}
\end{proof}

\begin{corollary}\label{cor: rigid assoc Fino-Raffero}
The submanifold $Y$ of Example \ref{lie_associative_subgroup} is rigid. 
\begin{proof}
Notice that the operator $\mathcal{B}$ vanishes on $Y$, as can be seen from
\begin{eqnarray*}
\mathcal{B}(\sigma)&=&\sum_{i,j=5}^7 (e_i\times e_j)\times B(e_j,S_\sigma(e_i))\\
                   &=&\sum_{k=1}^4\sum_{i,j=5}^7\langle S_{e_k}(e_j),S_\sigma(e_i)\rangle (e_i\times e_j)\times e_k\\
                   &=&-\sum_{k=1}^4\sum_{i,j,l=5}^7 \Gamma_{jk}^l\langle e_l, S_\sigma(e_i) \rangle (e_i\times e_j)\times e_k=0,
\end{eqnarray*}
since, $\Gamma_{jk}^l=0$ for  $j,l=5,6,7$ and $k=1,...,4$. Applying equation \eqref{fueter_dirac_formula}, Lemma \ref{alternative_weitzenbock_formula} and the previous calculation, we obtain the Weitzenböck formula
$$
  \fueter^2(\sigma)=\nabla^\ast\nabla\sigma+e_7\times\dirac(\sigma)-3\nabla_7^\perp\sigma-\frac{1}{2}\sigma.
$$
Taking the inner product with $\sigma$ and integrating over $Y$, 
\begin{align*}
\int_Y\langle\fueter^2(\sigma),\sigma\rangle d\vol_Y=&\int_Y\langle\nabla^\ast\nabla\sigma,\sigma\rangle d\vol_Y+\int_Y\langle e_7\times\dirac(\sigma),\sigma\rangle d\vol_Y-\int_Y3\langle \nabla_7^\perp\sigma,\sigma\rangle d\vol_Y\\
                                                     &-\int_Y\frac{1}{2}\langle\sigma,\sigma\rangle d\vol_Y\\
                                                    \geq& \int_Y\langle e_7\times\dirac(\sigma),\sigma\rangle d\vol_Y-3\int_Y \langle\dot{\sigma},\sigma\rangle d\vol_Y-\int_Y\frac{1}{2}\langle\sigma,\sigma\rangle d\vol_Y.
\end{align*}
From Lemma \ref{support_lemma } (iv), we conclude that 
\begin{align}\label{weitzenbock_example}
        \int_Y\langle\fueter^2(\sigma),\sigma\rangle d\vol_Y\geq& \int_Y\langle e_7\times\dirac(\sigma),\sigma\rangle d\vol_Y-\frac{1}{2}\int_Y\langle\sigma,\sigma\rangle d\vol_Y.
\end{align}
So, for $\sigma\in \ker\fueter$, we have $\dirac(\sigma)=-e_7\times \sigma$ and, replacing that in \eqref{weitzenbock_example}, we get the inequality
$$
  0\geq -\int_Y\langle e_7\times(e_7\times\sigma),\sigma\rangle d\vol_Y-\frac{1}{2}\int_Y\langle\sigma,\sigma\rangle d\vol_Y=\frac{1}{2}\int_Y\langle\sigma,\sigma\rangle d\vol_Y.
$$
Then $\sigma=0$ and therefore $Y$ is rigid. 
\end{proof}
\end{corollary}

\section*{Afterword}

In many cases a Weitzenböck formula is a useful tool to rule out parallel spinors, but in full generality equation \eqref{fueter_dirac_formula} has the drawback of first order terms with unpredictable spectrum. In  the nearly parallel case \eqref{nearly_parallel_case}, however, the Weitzenböck formula is very similar to the formula for a parameterized connection with skew-torsion symmetric tensor of  Agricola and Friedrich \cite{friedrich2004holonomy}. In this context and under favourable assumptions, it is possible to control the spectrum of $\dirac$. Using the Weitzenböck formula of \cite{gayet2014smooth}, we have 
$$
  \fueter^2=\nabla^\ast\nabla+\mathcal{R}+\mathcal{B}+\tau_0\dirac+\frac{\tau_0^2}{4}
$$
so, when a normal section lies in $\ker \fueter$, it corresponds to the eigenvalue $\frac{\tau_0}{4}$ of $\dirac$. Therefore, nontrivial deformations for an associative submanifold are in direct correspondence with elements of that eigenspace.


\end{document}